\documentclass[preprint,12pt]{elsarticle}
\usepackage{latexsym,amssymb,amsmath,amsthm,amsfonts,graphicx,float}
\usepackage{stix}
\usepackage{tikz}
\allowdisplaybreaks
\usepackage{hyperref}
\usepackage{array}
\usepackage{colortbl}
\usepackage{hhline}
\usepackage{soul}
\usepackage{color}
\usepackage[english]{babel}	
\usepackage{enumitem}
\usepackage{listings}
\tikzstyle{vertex}=[circle, draw, inner sep=0pt, minimum size=4pt]
\newcommand{\vertex}{\node[vertex]}

\usetikzlibrary{decorations.markings}
\tikzset{->-/.style={decoration={
  markings,
  mark=at position #1 with {\arrow{>}}},postaction={decorate}}}
  
\newcommand{\meet}{\wedge}
\newcommand{\join}{\vee}
\newcommand{\rtr}{\mathrm{rtr}}
\newcommand{\ltr}{\mathrm{ltr}}
\newcommand{\tr}{\mathrm{tr}}
\newcommand{\meetAss}{\mathrm{\meet\text{-}ass}}
\newcommand{\joinAss}{\mathrm{\join\text{-}ass}}
\newcommand{\ass}{\mathrm{ass}}
\newcommand{\dis}{\mathrm{dis}}
\newcommand{\D}{\mathrm{D}}
\newcommand{\Z}{\mathrm{Z}}
\newcommand{\cg}{\cellcolor{lightgray}}

\newtheorem{theorem}{Theorem}[section]

\newtheorem{corollary}{Corollary}[section]
\newtheorem{proposition}{Proposition}[section]

\theoremstyle{definition}
\newtheorem{definition}{Definition}[section]
\newtheorem{example}{Example}[section]
\newtheorem{remark}{Remark}[section]
\newtheorem{notation}{Notation}[section]

\begin{document}

\title{Triangular norms on bounded trellises}

\author[a,b]{Lemnaouar Zedam\corref{co1}}
\ead{lemnaouar.zedam@ugent.be}
\author[b]{Bernard De Baets}
\ead{bernard.debaets@ugent.be}
\cortext[co1]{Corresponding author}
\address[a]{LMPA, Department of Mathematics, University of M’sila, 28000 M'sila, Algeria}
\address[b]{KERMIT, Department of Data Analysis and Mathematical Modelling, Ghent University, Coupure links 653, B-9000 Gent, Belgium}

\begin{abstract}
In this paper, we introduce the notion of a t-norm on bounded pseudo-ordered sets and in particular on bounded trellises (also known as weakly associative lattices), and provide some basic examples.
The impact of abandoning transitivity is considerable: on a proper bounded trellis, the meet operation is not a t-norm, and there might actually exist no or even multiple maximal t-norms. We provide a first generic construction method that allows to extend a t-norm on an interior range of a given $\meet$-semi-trellis to the entire $\meet$-semi-trellis. Also, 
we discuss at length an instantiation of this method based on a particular interior range, namely a finite sub-trellis of the set of right-transitive elements of a given trellis. We pay specific attention to bounded pseudo-chains and modular trellises.
\end{abstract}
\maketitle

\section{Introduction}
The study of triangular norms (t-norms, for short) has a long history. These mathematical operations were introduced as early as the 1940s by Menger~\cite{Menger1942} when generalizing the triangle inequality from classical metric spaces to statistical metric spaces (nowadays called probabilistic metric spaces). In such spaces, distances 
are no longer non-negative real numbers, but are described by distribution functions instead.
The axiomatic description of t-norms as used today is due to Schweizer and  Sklar (see, e.g.,~\cite{SchweizerSklar1963,SchweizerSklar1983}):
a t-norm is an increasing, commutative and associative binary operation on the unit interval with $1$ as neutral element. The key property here is associativity: it allows a t-norm to be extended to any number of arguments
in an unambiguous way. From an algebraic point of view, the unit interval $[0,1]$ equipped with a t-norm can be seen as a totally ordered commutative monoidal structure. There is deep and extensive knowledge on t-norms, such as the
ordinal sum characterization of continuous t-norms by Ling~\cite{Ling1965}.

Soon after it was realized that t-norms can also serve the role of the logical connective `and' in fuzzy set theory~\cite{AnthonySherwood1979,Dubois1980}, indispensable for defining the intersection of fuzzy sets, they witnessed a second boom. Researchers in the field of fuzzy set theory replaced the hitherto used standard minimum operation by a t-norm, and this in virtually any theoretical development (such as Zadeh's extension principle~\cite{Kerre2011,Zadeh1975}) or practical application (like in fuzzy optimization). Also mathematicians got involved, further expanding the knowledge on these operations, culminating in the seminal book
of Klement, Mesiar and Pap~\cite{Klement2013}. Undoubtedly, they also marked the inception of another flourishing subfield of fuzzy set theory --- and data science at large --- namely the theory of aggregation functions (see, e.g., \cite{Beliakov2007,Calvo2002,Grabisch2009}). Of similar importance and sharing a long history are the binary operations called copulas introduced by Sklar~\cite{Sklar1959}. A copula is a bivariate cumulative distribution function for which 
the marginal probability distribution of each variable is uniform on the unit interval $[0, 1]$. They are not associative, thus requiring appropriate definitions for higher dimensions. The latter problem disappears in the case of associative copulas. Interestingly, associative copulas are nothing else but 1-Lipschitz continuous t-norms~\cite{Nelsen2007}, explaining another wave of cross-fertilization between probability theory (studying copulas) and fuzzy set theory (studying t-norms).

Soon after Zadeh's seminal paper on fuzzy set~\cite{Zadeh1965}, it was already realized by Goguen~\cite{Goguen1967} that the lattice-theoretic setting (see, e.g.,~\cite{Birkhoff1967,DaveyPriestley2002}) is the most natural one for the development of fuzzy set theory. Notwithstanding the many achievements mentioned above, all situated in the comfortable setting of real analysis, it took until the 90s for the study of t-norms on more general structures to take off~\cite{BaetsMesiar1999,DeCooman1994}, in particular on partially ordered sets (posets, for short) and lattices, thus being related to the study of partially ordered semigroups~\cite{Fuchs2011}. Not surprisingly, research in this direction has been a lot slower, initially being focused on construction methods~\cite{Asici2014,Asici2016,Cayli2019,Karacal2020,Kesiciouglu2013}, and only recently more profound characterizations in terms of ordinal sums~\cite{DvorakHolcapek2020,Ouyang2021,Ouyang2022}.

Note that in a bounded lattice $(L, \leq, \meet, \join,0,1)$, the transitivity of the partial order relation $\leq$ and/or the associativity of the meet operation $\meet$ and join operation $\join$ play an important role.

Characteristic of the partial order relation of a poset or lattice is indeed its transitivity property, by far the most fascinating relational property. Above all, transitivity is a convenient property, simplifying mathematical reasoning until it becomes an automatism. However, many theoretical and practical developments warrant us to look beyond transitivity. For instance, the study of (absence of) transitivity in the comparison of random variables~\cite{BaetsMeyer2015,BaetsMeyer2006} has proven to more than intriguing and resulted in the framework of cycle-transitivity. Indeed, absence of transitivity can manifest itself in two flavors: 
the presence of cycles or preference loops (A better than B, B better than C, and C better than A) or simply incomparability (A better than B, B better than C, but A and C being incomparable). Although often inconvenient to mathematicians and computer scientists (see the interest in acyclic directed graphs, for instance), cycles have been shown to be of extreme importance in practical situations, e.g.~cycles in species competition structures such as tournaments preventing extinction and supporting biodiversity~\cite{Kerr2002,Reichenbach2007}, usually catalogued under the Rock-Paper-Scissors metaphor. Also incomparibility is a well-known phenomenon, such as the intransitivity of indifference~\cite{Fishburn1970}.

Motivated by the interest in non-transitive relations as well as our unrelenting mathematical curiosity, in this paper we set out to introduce and study t-norms on lesser-known, albeit interesting mathematical structures, namely the class of pseudo-ordered sets (psosets, for short), and in particular the subclass of trellises~\cite{Skala1971}. The latter are also known as tournament lattices~\cite{Fried1970}, non-associative lattices~\cite{Donald1971} or weakly associative lattices~\cite{Chajda1977,Fried1973b, Fried1975}. However, we prefer the designated technical term `trellis'. Psosets generalize posets by replacing the partial order relation by a more general reflexive and antisymmetric relation, while trellises do the same compared to lattices, but preserve the existence of meets and joins, {\em i.e.}~the existence of greatest lower bounds and smallest upper bounds of 2-element subsets~\cite{Chajda1977,Fried1973b, Fried1975}). However, the absence of transitivity of the pseudo-order relation is reflected in the absence of the associativity property of the meet and join operation of a trellis. Major contributions to the study of trellises have been made by Skala~\cite{Skala1971,Skala1972}, followed by several other scholars. 
For instance, Gladstien~\cite{Gladstien1973} proved that trellises of finite length are complete if and only if every cycle has a smallest upper bound and a greatest lower bound. Later on, this characterization was generalized to psosets in terms of joins of cycles and pseudo-chains~\cite{Bhatta2004,Rai2021a,Rai2021b}.

Our goal is to unravel how abandoning the transitivity property affects the notion of a t-norm on bounded psosets, and in particular on bounded trellises. The remainder of this paper is structured as follows. In the preliminary Section~\ref{Basic_concepts}, we recall the basic notions on psosets and trellises needed in this paper. We identify some potentially interesting subsets of a trellis in Section~\ref{Specific subsets of a trellis}. In Section~\ref{t-norms_on_psosets}, we extend the notion of a t-norm to the setting of bounded psosets and bounded trellises, and provide some examples. Our main contribution is to be found in Sections~\ref{Int_based-construction_of_t-norms} and~\ref{T-norms_based_on_lambda_A}, in which we present a first generic construction method based on a t-norm on an interior range of a given $\meet$-semi-trellis. We illustrate the method for a particular
interior operator with as range an appropriate finite subset of the set of right-transitive elements of a given trellis.
Finally, we present some concluding remarks and future research lines in Section~\ref{Conclusion}.

\section{Basic concepts}\label{Basic_concepts}
This section serves an introductory purpose. First, we recall some definitions and properties related to pseudo-ordered sets and trellises. Second, we present some specific elements of a trellis that will be needed throughout this paper.
 
\subsection{Pseudo-ordered sets and trellises}
In this subsection, we recall the notions of pseudo-ordered sets and trellises; more information can be found in~\cite{Fried1970,Skala1971,Skala1972}.
A \textit{pseudo-order (relation)} $\unlhd$ on a set $X$ is a binary relation on $X$ that is reflexive ({\em i.e.}, $x \unlhd x$, 
for any $x \in X$) and antisymmetric ({\em i.e.}, $x \unlhd y$ and $y \unlhd x$ implies $x = y$, for any $x,y \in X$). A set $X$ equipped
with a pseudo-order $\unlhd$ is called a \textit{pseudo-ordered set} (\textit{psoset}, for short) and is denoted by $\mathbb{P}=(X,\unlhd)$; a psoset that is not a poset is called a proper psoset.
For any two elements $a, b \in X$, if $a \unlhd b$ and $a\neq b$, then we write $a \lhd b$; if $a \unlhd b$ does not hold, then we also write $a \ntrianglelefteq b$. For any element $x\in X$, we introduce the following subsets of~$X$:
\[
\downarrow x=\{y\in X \mid y\unlhd x\}
\qquad \text{and} \qquad \uparrow x=\{y\in X \mid x\unlhd y\}\,.
\]

For a given psoset $\mathbb{P}=(X,\unlhd)$ and $x,y\in X$, we write $x\lesssim y$ if there exists a finite sequence $(x_{1}, \ldots, x_{n})$ of elements from $X$ such that $x \unlhd x_{1} \unlhd \ldots\unlhd x_{n} \unlhd y$. Note that the relation $\lesssim$ is a pre-order relation, {\em i.e.}, it is reflexive and transitive, but not necessarily antisymmetric.

Consider a subset $\mathcal{C}$ of $X$ and $x,y \in \mathcal{C}$. We write $x \lesssim_\mathcal{C} y $ if there exists a finite sequence $(x_{1}, \ldots, x_{n})$ of elements from $\mathcal{C}$ such that $x \unlhd x_{1} \unlhd \ldots\unlhd x_{n} \unlhd y$.  If for any $x,y \in \mathcal{C}$, $x \lesssim_\mathcal{C} y $ or $y \lesssim_\mathcal{C} x $ holds, then $\mathcal{C}$ is called a \textit{pseudo-chain}.
If for any $x,y \in \mathcal{C}$, both $x \lesssim_\mathcal{C} y $ and $y \lesssim_\mathcal{C} x $ hold,
then $\mathcal{C}$ is called a \textit{cycle}.
Note that every singleton $C=\{x\}$ is a trivial pseudo-chain as well as a trivial cycle. Due to the antisymmetry of $\unlhd$, any non-trivial cycle contains at least three elements.

Similarly as for partially ordered sets (posets, for short), a finite pseudo-ordered set can be represented by 
a \textit{Hasse-type diagram} (or, simply, \textit{Hasse diagram}) with the following difference: if $x$ and $y$ are not related, while in a poset set this would be implied by transitivity, then $x$ and $y$ are joined by a dashed edge. If $x\lesssim y$ and $y \unlhd x$, then $x$ and $y$ are joined by a directed edge going from $y$ to $x$.  

\begin{example}\label{example psoset} 
Let $\mathbb{P}=(\{a,b,c,d,e,f\}, \unlhd)$ be the psoset given by the Hasse diagram in Fig.~\ref{Fig01} and pseudo-order $\unlhd$
in Table~\ref{PseudoWithCycle}. Here, $b\unlhd d$, $d\unlhd e$, while $b\ntrianglelefteq e$. Also, $\{d,e,f\}$ is a cycle. 

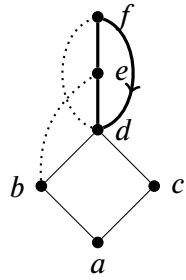
\begin{figure}[h]
\[\begin{tikzpicture}
\tikzstyle{estun}=[>=latex,thick,dotted]
    \vertex[fill] (a) at (0,0)  [label=below:$a$]  {};
    \vertex[fill] (b) at (-0.75,0.75)  [label=left:$b$]  {};
    \vertex[fill] (c) at (0.75,0.75)  [label=right:$c$]  {};
    \vertex[fill] (d) at (0,1.5)  [label=right:$d$]  {};
    \vertex[fill] (e) at (0,2.25)  [label=right:$e$]  {};
    \vertex[fill] (f) at (0,3)  [label=right:$f$]  {};
   
    \path
        (a) edge (b)
        (a) edge (c)
        (b) edge (d)
        (c) edge (d)
        ;
        \draw [very thick] (d) to (e);
        \draw [very thick] (e) to (f)
            
        ;
   \draw[estun] (b) to [bend left=30] (e);
   \draw[estun] (d) to [bend left=70] (f) 
   
   ;
   
  \draw[very thick, ->-=0.65] (f) to [bend left=70] (d)
 
           ;
\end{tikzpicture}\]
\caption{Hasse diagram of the psoset in Example~\ref{example psoset}.} \label{Fig01}
\end{figure}
\begin{table}[H]
 \begin{center}
 \begin{tabular}{|c|c|c|c|c|c|c|}
\hline 
$\unlhd$ & $a$ & $b$ & $c$ & $d$ & $e$ & $f$ \\
\hline   
$a$	     &$1$ & $1$& $1$ & $1$ & $1$ & $1$ \\ 
\hline 
$b$	     &$0$ & $1$& $0$ & $1$ & $0$ & $1$ \\ 
\hline 
$c$      &$0$ & $0$& $1$ & $1$ & $1$ & $1$ \\ 
\hline 
$d$      &$0$ & $0$& $0$ & $1$ & $1$ & $0$ \\
\hline 
$e$      &$0$ & $0$& $0$ & $0$ & $1$ & $1$ \\ 
\hline 
$f$      &$0$ & $0$& $0$ & $1$ & $0$ & $1$ \\ 
\hline 
\end{tabular}
    \caption{Pseudo-order of the bounded psoset in Example~\ref{example psoset}.}
    \label{PseudoWithCycle}
    \end{center}
\end{table}
\end{example}

The notions of \textit{minimal/maximal} element, \textit{smallest/greatest} element, \textit{lower/upper bound}, 
\textit{greatest lower bound} (or \textit{infimum}), 
\textit{smallest upper bound} (or \textit{supremum}) for psosets are defined in the same way as the corresponding notions for posets. For a subset $A$ of a psoset $\mathbb{P}=(X,\unlhd)$, the antisymmetry of the pseudo-order implies that if $A$ has an infimum (resp.\ supremum), then it is unique, and is denoted by $\bigwedge A$ (resp.\ $\bigvee A$). If $A=\{a, b\}$, 
then we  write $a\meet b$ (called \textit{meet}) instead of $\bigwedge \{a, b\}$ and $a\join b$ (called \textit{join}) instead of $\bigvee \{a, b\}$. A \textit{bounded psoset} is a psoset that has a smallest element denoted by $0$ and a greatest element denoted by $1$, {\em i.e.},~$0 \unlhd x \unlhd 1$, for any $x \in X$. A bounded psoset $\mathbb{P}=(X,\unlhd)$ is denoted by $\mathbb{P}=(X,\unlhd,0,1)$. 

\begin{definition}\cite{Gladstien1973}
A $\meet$-semi-trellis (resp.\ $\join$-semi-trellis) is a psoset $\mathbb{P}=(X,\unlhd)$ such that $x \meet y$
(resp.~$x \join y$) exists for all $x,y \in X$. A \textit{trellis} is a psoset that is both a $\meet$-semi-trellis and a $\join$-semi-trellis; it is denoted by $\mathbb{T}=(X, \unlhd, \meet, \join)$; a trellis that is not a lattice is called a proper trellis.
\end{definition}

A trellis can also be characterized as an algebra $(X, \meet, \join)$, where $X$ is a nonempty set and the binary operations $\meet$ and $\join$ satisfy the following properties,
for any $a,b,c\in X$~\cite{Skala1971}:
\begin{enumerate}[label=(\roman*)]
\item $a \join b = b \join a$ and $a \meet b = b \meet a$ (\textit{commutativity})\,;
\item $a \join( b \meet a) = a= a \meet (b \join a)$ (\textit{absorption})\,;
\item $a \join( (a \meet b) \join ( a \meet c)) = a= a \meet( (a \join b) \meet ( a \join c)) $ (\textit{part-preservation})\,.
\end{enumerate}

\begin{theorem}{\rm \cite{Skala1971}}
A set $X$ with two commutative, absorptive, and part-preserving operations $\meet$ and $\join$  is a trellis if
the pseudo-order $\unlhd$ is defined as follows: $a \unlhd b$ if $a \meet b=a$ and/or $a \join b=b$. 
The operations are also idempotent ({\em i.e.}, $x \meet x= x \join x=x$, for any $x \in X$).
\end{theorem}

One observes that the difference between the notions of a lattice and a trellis is that the operations $\meet$ and $\join$ are not required to be associative in the case of a trellis.

A trellis $\mathbb{T}=(X, \unlhd, \meet, \join)$ is called \textit{bounded} if it is bounded as a psoset.  The notation $\mathbb{T}=(X, \unlhd, \meet, \join,0,1)$ is used for a bounded trellis. Also, a trellis is called \textit{complete} if every 
subset has an infimum and a supremum.

\begin{definition} \cite{Gladstien1973}
Let $\mathbb{T}=(X,\unlhd,\meet,\join)$ be a trellis and $A \subseteq X$. Then
\begin{enumerate}[label=(\roman*)]
\item $A$ is called a \textit{$\meet$-sub-trellis} of~$\mathbb{T}$ if $x \meet y \in A$, for any $x,y \in A$;
\item $A$ is called a \textit{$\join$-sub-trellis} of~$\mathbb{T}$ if $x \join y \in A$, for any $x,y \in A$;
\item $A$ is called a \textit{sub-trellis} of~$\mathbb{T}$ if it is both a $\meet$-sub-trellis and a 
$\join$-sub-trellis of~$\mathbb{T}$;
\item $A$ is called a \textit{sub-lattice} of~$\mathbb{T}$ if it is a sub-trellis of~$\mathbb{T}$ and 
$\unlhd$ is transitive on~$A$.
\end{enumerate}
\end{definition}

\begin{theorem}{\rm \cite{Skala1972}}\label{Transitivity_equivalent_to_Meet_join_ass}
Let $\mathbb{T}=(X,\unlhd,\meet,\join)$ be a trellis. The following statements are equivalent:
\begin{enumerate}[label=(\roman*),font=\upshape]  
\item $\unlhd$ is transitive;
\item both $\meet$ and $\join$ are associative; 
\item one of $\meet$ and $\join$ is associative.
\end{enumerate}
\end{theorem}

\begin{definition}\cite{Skala1972}\label{Modular_Trellis}
A trellis $\mathbb{T}=(X,\unlhd,\meet,\join)$ is said to be \textit{modular}, if for any $x,y,z \in X$,
it holds that $x \unlhd z$ implies $x \join (y \meet z)= (x \join y ) \meet z$.
\end{definition}

\begin{remark}\label{ModularNon3-cycles}
A modular trellis does not include any cycle of three elements. Indeed, suppose that a modular trellis $\mathbb{T}=(X,\unlhd,\meet,\join)$ includes a cycle of three distinct elements $\mathcal{C}=\{x_{1},x_{2},x_{3}\}$. This implies that $x_i \lesssim_\mathcal{C} x_j$, for any $i,j\in \{1,2,3\}$ with $i\neq j$. That means that exactly one of the following statements holds: $x_{1} \unlhd x_{2}\unlhd x_{3}\unlhd x_{1}$; $x_{1} \unlhd x_{3}\unlhd x_{2}\unlhd x_{1}$; $x_{2} \unlhd x_{1}\unlhd x_{3}\unlhd x_{2}$; $x_{2} \unlhd x_{3}\unlhd x_{1}\unlhd x_{2}$; $x_{3} \unlhd x_{1}\unlhd x_{2}\unlhd x_{3}$ or $x_{3} \unlhd x_{2}\unlhd x_{1}\unlhd x_{3}$. 
Any of these statements contradicts the modularity of $\mathbb{T}$. Suppose for instance that $x_{1} \unlhd x_{2}\unlhd x_{3}\unlhd x_{1}$. Since $x_{1} \unlhd x_{2}$, it holds by the modularity of $\mathbb{T}$ that $x_{1} \join (x_{3} \meet x_{2})= (x_{1} \join x_{3} ) \meet x_{2}$. Hence, $x_{2}=x_{1}$, a contradiction. 
\end{remark}

\subsection{Specific elements of a trellis}
In this subsection, we present some specific elements of a trellis that will play an important role in this paper.

\begin{definition}\cite{Skala1972}
Let $\mathbb{P}=(X,\unlhd)$ be a psoset. An element $a \in X$ is called:
\begin{enumerate}[label=(\roman*)]  
\item \textit{right-transitive}, if $a \unlhd x \unlhd y$ implies  $a \unlhd y$, for any $x,y \in X$; 
\item \textit{left-transitive}, if $x \unlhd y \unlhd a$ implies $x \unlhd a$, for any $x,y \in X$;
\item \textit{middle-transitive}, if $x \unlhd a \unlhd y$ implies $x \unlhd y$, for any $x,y \in X$;
\item \textit{transitive}, if it is right-, left- and middle-transitive.
\end{enumerate}
\end{definition}

\begin{definition}\cite{Skala1972}\label{assoc}
Let $\mathbb{T}=(X,\unlhd,\meet,\join)$ be a trellis. 
\begin{enumerate}[label=(\roman*)]  
\item A 3-tuple $(x,y,z) \in X^3$ is called  $\meet$-\textit{associative} (resp.\ $\join$-\textit{associative}) if $(x \meet y )\meet z = x\meet(y \meet z)$ (resp.\ $(x\join y)\join z = x\join(y \join z)$).
\item An element  $a \in X$ is called $\meet$-associative (resp.\ $\join$-associative) if any 3-tuple including $a$ is $\meet$-associative (resp.\ $\join$-associative). 
 \item An element $a\in X$ is called \textit{associative} if it is both $\meet$- and $\join$-associative. 
\end{enumerate}
\end{definition}

Note that due to the commutativity of $\meet$ and $\join$, it is sufficient to consider 3-tuples 
of the type $(a,x,y)$ in Definition~\ref{assoc}(ii).

For further use, we recall some results about transitive and associative elements. 

\begin{proposition}{\rm \cite{Skala1972}}\label{associative element is transitive}
Let $\mathbb{T}=(X,\unlhd,\meet,\join)$ be a trellis. Any $\meet$-associative or $\join$-associative element is transitive, but the converse does not hold in general.
\end{proposition}

\begin{theorem}{\rm \cite{Skala1972}}\label{Ass=tr}
Let $\mathbb{T}=(X,\unlhd,\meet,\join)$ be a trellis. If $\mathbb{T}$ is a modular trellis or a pseudo-chain, then any element is associative if and only if it is transitive. 
\end{theorem}

\begin{proposition}{\rm\cite{Skala1972}}\label{theoremoftransitiveelement}
Let $(X,\unlhd,\meet,\join)$ be a trellis and $a\in X$.  
\begin{enumerate}[label=(\roman*),font=\upshape] 
\item If $a$ is right-transitive, then
\begin{enumerate}
\item[(a)] $a \lesssim x$ implies $a \unlhd x$, for any $x\in X$;
\item[(b)] $a\unlhd x $ implies $a \join y \unlhd x \join y$, for any $x,y\in X$.
\end{enumerate}
\item If $a$ is left-transitive, then
\begin{enumerate}
\item[(a)] $x\lesssim a$ implies $x \unlhd a$, for any $x\in X$;
\item[(b)] $x\unlhd a$ implies $x\meet y\unlhd a \meet y$, for any $x,y\in X$.
\end{enumerate}
\item If $a$ is $\meet$-associative, then $x \unlhd y$ implies $ a \meet x \unlhd a \meet y$, for any $x,y\in X$.
\item If $a$ is $\join$-associative, then $x \unlhd y$ implies $ a \join x \unlhd a \join y$, for any $x,y\in X$.
\end{enumerate}
\end{proposition}

\begin{proposition}
Note that if a cycle $\mathcal{C}$ of a psoset $\mathbb{P}=(X,\unlhd)$ contains a right- or a left-transitive element, then it is a trivial cycle, {\em i.e.}, $\mathcal{C}$ is a singleton subset.
\end{proposition}

\begin{proof}
Let $x$ be a right-transitive element of $\mathcal{C}$ and $y\in \mathcal{C}$. We prove that $y=x$. Indeed, there exist two finite sequences $(x_1,..., x_n)$ and $(y_1,..., y_m)$ in $\mathcal{C}$ such that $x\unlhd x_{1} \unlhd x_{2}\unlhd \ldots \unlhd x_{n} \unlhd y \unlhd y_{1} \unlhd y_{2}\unlhd \ldots \unlhd y_{m} \unlhd x$.
From Proposition~\ref{theoremoftransitiveelement}(i) and the antisymmetry of $\unlhd$, it follows that $y_m=x$. Inductively, it follows that $y=x$. The proof is similar for a left-transitive element.
\end{proof}

 \begin{definition}\cite{Skala1972}
Let $\mathbb{T}=(X,\unlhd,\meet,\join)$ be a trellis. 
\begin{enumerate}[label=(\roman*)]  
\item A 3-tuple $(x,y,z) \in X^3$ is called  $\meet$-\textit{distributive} (resp.\ $\join$-\textit{distributive}) if $(x \meet y) \join z = (x\join z)\meet (y\join z)$ (resp.\ $(x \join y) \meet z = (x\meet z)\join (y\meet z)$);
\item An element  $a \in X$ is called $\meet$-distributive (resp.\ $\join$-distributive) if any 3-tuple including $a$ is $\meet$-distributive (resp.\ $\join$-distributive). 
\end{enumerate}
\end{definition}

Note that an element is $\meet$-distributive if and only if it is $\join$-distributive.

\begin{definition}\cite{Skala1972}\label{Distributive_Trellis}
Let $\mathbb{T}=(X,\unlhd,\meet,\join)$ be a trellis. 
\begin{enumerate}[label=(\roman*)]  
\item An element $a \in X$ is called distributive if it is $\meet$-distributive or $\join$-distributive;
\item A trellis $\mathbb{T}=(X,\unlhd,\meet,\join)$ is said to be \textit{distributive} if all its elements are distributive. 
\end{enumerate}
\end{definition}

The following results show that there are no proper distributive trellises. 

\begin{proposition}{\rm \cite{Fried1973b}}
Any distributive trellis is a lattice.
\end{proposition}

\begin{proposition}{\rm \cite{Skala1972}}\label{DisElemts_is_AssElmts}
Any distributive element of a trellis is associative, 
but the converse does not hold in general.
\end{proposition}

\section{Specific subsets of a trellis}\label{Specific subsets of a trellis}
In this section, we introduce our notations for some interesting subsets of a trellis.

 \begin{notation}
Let $\mathbb{T}=(X,\unlhd,\meet,\join)$ be a trellis. We denote  by:
\begin{enumerate}[label=(\roman*)] 
\item $X^\rtr$:  the set of right-transitive elements of $X$;
\item $X^\ltr$:  the set of left-transitive elements of $X$;
\item $X^\tr$: the set of transitive elements of $X$;
\item $X^\meetAss$:  the set of $\meet$-associative elements of $X$;
\item $X^\joinAss$:  the set of $\join$-associative elements of $X$;
\item $X^\ass$: the set of associative elements of $X$;
\item $X^\dis$:  the set of distributive elements of $X$.
\end{enumerate}  
\end{notation}

It is easy to see that if $\mathbb{T}=(X,\unlhd,\meet,\join,0,1)$ is a bounded trellis, then $(X^\alpha,\unlhd,0,1)$ is a bounded psoset, for any $\alpha\in\{\dis,\ass,\meetAss,\joinAss,\tr,\ltr,\rtr\}$. 

The following proposition is an immediate consequence of Propositions~\ref{associative element is transitive} and~\ref{DisElemts_is_AssElmts}. 

\begin{proposition}\label{Specific_sets_inclusion}
Let $\mathbb{T}=(X,\unlhd,\meet,\join)$ be a trellis. The following chains of inclusions hold:
\begin{enumerate}[label=(\roman*),font=\upshape] 
\item $X^\dis \subseteq X^\ass\subseteq X^\meetAss\subseteq X^\tr\subseteq X^\rtr$;
\item $X^\dis \subseteq X^\ass\subseteq X^\joinAss\subseteq X^\tr\subseteq X^\ltr$.
\end{enumerate}
\end{proposition}

Recall that $X^\rtr$ and $X^\ltr$ do not include any non-trivial cycle. 

The following corollary is an immediate consequence of Proposition~\ref{Specific_sets_inclusion} and Theorem~\ref{Ass=tr}. 

\begin{corollary}\label{Eq_subsets_pseud-chain_and_modular}
Let $\mathbb{T}=(X,\unlhd,\meet,\join)$ be a trellis. If $\mathbb{T}$ is modular or a pseudo-chain, then the following chain of equalities holds:
\[X^\ass= X^\meetAss = X^\joinAss=X^\tr\,.\] 
\end{corollary}

The following proposition recalls some results about the
lattice structure of some of the above subsets.
 
\begin{proposition}{\rm\cite{Skala1972}}\label{X^tr=X^ass_X^dis_sublattice}
Let $\mathbb{T}=(X,\unlhd,\meet,\join)$ be a trellis. It holds that:
\begin{enumerate}[label=(\roman*),font=\upshape] 
\item The subset $X^\dis$ is a distributive sub-lattice;
\item If $\mathbb{T}$ is modular or a pseudo-chain, then the subsets $X^\tr=X^\ass$ are sub-lattices.
\end{enumerate}  
\end{proposition}

The following proposition provides some results about the trellis structure of some of the above subsets. The proof is an immediate application of Proposition~\ref{theoremoftransitiveelement}.

\begin{proposition}\label{left-_right-transitive_elements_are_meet-semi-trellis_join-semi-trellis}
Let $\mathbb{T}=(X,\unlhd,\meet,\join)$ be a trellis. It holds that:
\begin{enumerate}[label=(\roman*),font=\upshape] 
\item $(X^\ltr,\unlhd,\meet)$ is a $\meet$-sub-trellis of~$\mathbb{T}$;
\item $(X^\rtr,\unlhd,\join)$ is a $\join$-sub-trellis of~$\mathbb{T}$;
\item $(X^\meetAss,\unlhd,\meet)$ is a $\meet$-sub-trellis of~$\mathbb{T}$; 
\item $(X^\joinAss,\unlhd,\join)$ is a $\join$-sub-trellis of~$\mathbb{T}$.
\end{enumerate}
\end{proposition}

\begin{proposition}\label{A_subtrellis_on_pseudo-chain}
Let $\mathbb{T}=(X,\unlhd,\meet,\join)$ be a trellis and $A$ a subset of $X$. If $\mathbb{T}$ is a pseudo-chain and $A\subseteq X^\rtr$ or $A\subseteq X^\ltr$, then $A$ is a sub-trellis of~$\mathbb{T}$.
\end{proposition}

\begin{proof}
 Let $x,y\in A$. There exists a finite sequence $\left(a_{1}, \ldots, a_{n}\right)$ of elements in $X$ such that $x \unlhd a_{1} \unlhd \ldots \unlhd a_{n} \unlhd y $ or $y \unlhd a_{1} \unlhd \ldots \unlhd a_{n} \unlhd x$. The fact that $A\subseteq X^\rtr$ or $A\subseteq X^\ltr$  implies that $x \unlhd  y $ or $y \unlhd x$. Thus, $x \meet y= x \in A$ or $x \meet y= y \in A$. In a similar way, we obtain that $x \join y= y \in A$ or $x \join y= x \in A$. Thus, $A$ is a sub-trellis of~$\mathbb{T}$.
 \end{proof}

The following proposition shows that any 3-tuple of elements of a subset of $X^\rtr$ (resp.\ of $X^\ltr$) is $\join$-associative (resp.\ $\meet$-associative).

\begin{proposition}\label{join_meet_ass_in_A}
Let $\mathbb{T}=(X,\unlhd,\meet,\join)$ be a trellis and $A$ a subset of $X$. The following equalities hold:
\begin{enumerate}[label=(\roman*),font=\upshape]
\item If $A\subseteq X^\rtr$, then  $x \join (y \join z)  = (x \join y) \join z$, for any $x,y,z \in A$;
\item If $A\subseteq X^\ltr$, then  $x \meet (y \meet z)  = (x \meet y) \meet z$, for any $x,y,z \in A$.
\end{enumerate}
\end{proposition}

\begin{proof}
We only give the proof of (i), the proof of (ii) being dual. Suppose that  $A\subseteq X^\rtr$, and let $x,y,z \in A$. We have $x \unlhd x \join y \unlhd (x \join y) \join z$, $y \unlhd x \join y \unlhd (x \join y) \join z$ and $ z \unlhd (x \join y) \join z$. Since $A\subseteq X^\rtr$, it follows that $x \unlhd  (x \join y) \join z$, $y  \unlhd (x \join y) \join z$ and $ z \unlhd (x \join y) \join z$. Moreover, $ y \join z \unlhd (x \join y) \join z$. Hence, $ x \join (y \join z) \unlhd (x \join y) \join z$. In a similar way, we obtain that  $(x \join y) \join z \unlhd x \join (y \join z) $. Thus, $  x \join (y \join z) = (x \join y) \join z$. 
\end{proof}

In view of Proposition~\ref{join_meet_ass_in_A}, one can deduce the following corollary. 
\begin{corollary}\label{subtrellis_implies_sublattice}
Let $\mathbb{T}=(X,\unlhd,\meet,\join)$ be a trellis and $A$ a subset of $X^\rtr$ or $X^\ltr$. 
\begin{enumerate}[label=(\roman*),font=\upshape]
\item If $A$ is a $\meet$-sub-trellis of~$\mathbb{T}$, then it is a $\meet$-sub-lattice of~$\mathbb{T}$;
\item If $A$ is a $\join$-sub-trellis of~$\mathbb{T}$, then it is a $\join$-sub-lattice of~$\mathbb{T}$;
\item If $A$ is a sub-trellis of~$\mathbb{T}$, then it is a sub-lattice of~$\mathbb{T}$. 
\end{enumerate}
\end{corollary}

\begin{corollary}\label{Specifics_subtrellis_implies_sublattice}
Let $\mathbb{T}=(X,\unlhd,\meet,\join)$ be a trellis and $\alpha\in\{\ass,\meetAss,\joinAss,\tr,\ltr,\rtr\}$. If $X^\alpha$ is a sub-trellis of~$\mathbb{T}$, then it is a sub-lattice of~$\mathbb{T}$. 
\end{corollary}
Combining Theorem~\ref{Transitivity_equivalent_to_Meet_join_ass} and Proposition~\ref{join_meet_ass_in_A} leads to the following result.
\begin{proposition}\label{conditions subsets of lattices}
Let $\mathbb{T}=(X,\unlhd,\meet,\join)$ be a trellis. The following statements are equivalent: 
\begin{enumerate}[label=(\roman*),font=\upshape]
\item $\mathbb{T}=(X,\unlhd,\meet,\join)$ is a lattice;
\item $X^\rtr =X$;
\item $X^\ltr =X$;
\item $X^\meetAss =X$; 
\item $X^\joinAss =X$.
\end{enumerate}
\end{proposition}

\begin{theorem}{\rm \cite{Skala1972}}\label{Closeness_of_join_for_rtr+joinass}
Let $\mathbb{T}=(X,\unlhd,\meet,\join)$ be a trellis.  
\begin{enumerate}[label=(\roman*),font=\upshape] 
\item If $x_{1},\ldots,x_{k}$ are right-transitive, then the join of $\left\{x_{1},\ldots,x_{k}\right\}$ exists and equals $x_{i_1} \join \ldots \join x_{i_{k}}$ for any permutation $i_{1}, \ldots, i_{k}$ of $1, \ldots, k$. Moreover, $\bigvee \left\{x_{1}, \ldots, x_{k}\right\}$ is right-transitive. 
\item If $x_{1},\ldots,x_{k}$ are left-transitive, then the meet of $\left\{x_{1},\ldots,x_{k}\right\}$ exists and equals $x_{i_1} \meet \ldots \meet x_{i_{k}}$ for any permutation $i_{1}, \ldots, i_{k}$ of $1, \ldots, k$. Moreover, $\bigwedge \left\{x_{1}, \ldots, x_{k}\right\}$ is left-transitive. 
\item If $x_{1}, \ldots, x_{k}$ are $\meet$-associative, then  the meet of $\left\{x_{1}, \ldots, x_{k}\right\}$ exists and equals $x_{i_1} \meet \ldots \meet x_{i_{k}}$ for any permutation $i_{1}, \ldots, i_{k}$ of $1, \ldots, k$. Moreover, $\bigwedge \left\{x_{1}, \ldots, x_{k}\right\}$ is $\meet$-associative. 
\item If $x_{1}, \ldots, x_{k}$ are $\join$-associative, then  the join of $\left\{x_{1}, \ldots, x_{k}\right\}$ exists and equals $x_{i_1} \join \ldots \join x_{i_{k}}$ for any permutation $i_{1}, \ldots, i_{k}$ of $1, \ldots, k$. Moreover, $\bigvee \left\{x_{1}, \ldots, x_{k}\right\}$ is $\join$-associative. 
\end{enumerate}  
\end{theorem}

The following result generalizes the above Theorem~\ref{Closeness_of_join_for_rtr+joinass}. The proof is straightforward.
\begin{proposition}\label{Closeness_of_join_for_subtrellis}
Let $\mathbb{T}=(X,\unlhd,\meet,\join)$ be a trellis and $A$ a subset of $X$. 
\begin{enumerate}[label=(\roman*),font=\upshape] 
\item If $A$ is a subset of $X^\rtr$, then for any finite subset $\{x_1, \ldots, x_k\}\subseteq A$, it holds that $\bigvee \{x_1, \ldots, x_k\}$ exists and equals $x_{i_1} \join \ldots \join x_{i_{k}}$ for any permutation $i_1, \ldots, i_k$ of $1, \ldots, k$. Moreover, if $A$ is $\join$-sub-trellis of~$\mathbb{T}$, then $\bigvee\{x_1, \ldots, x_k\}\in A$.
\item If $A$ is a subset of $X^\ltr$, then for any finite subset $\{x_1, \ldots, x_k\}\subseteq A$, it holds that $\bigwedge \{x_1, \ldots, x_k\}$ exists and equals $x_{i_1} \meet \ldots \meet x_{i_{k}}$ for any permutation $i_1, \ldots, i_k$ of $1, \ldots, k$. Moreover, if $A$ is $\meet$-sub-trellis of~$\mathbb{T}$, then $\bigwedge\{x_1, \ldots, x_k\}\in A$.
\end{enumerate}  
\end{proposition}

\section{Triangular norms on bounded psosets and trellises}\label{t-norms_on_psosets}
In this section, we extend the notion of a t-norm to the setting of bounded psosets and trellises and provide some examples. For further use, we recall several notions and list some properties of binary operations on psosets and trellises.

\subsection{Binary operations on psosets and trellises}
In this subsection, we present some basic definitions and properties of binary operations on a psoset or trellis. Some of them are adopted from the corresponding notions on a poset or lattice (see, e.g.,~\cite{DaveyPriestley2002,Fried1973b,Yettou2019}).
Consider a bounded psoset $\mathbb{P}=(X,\unlhd)$. A binary operation~$F$ on $\mathbb{P}$ is called:
\begin{enumerate}[label=(\roman*)]
\item \textit{commutative}, if $F(x,y)=F(y,x)$, for any $x,y \in X$;
\item \textit{associative}, if $F(x,F(y,z))=F(F(x,y),z)$, for any $x,y,z \in X$;
\item \textit{right-increasing}, if $x \unlhd y \text{ implies } F(z,x) \unlhd F(z,y)$, for any $x,y,z \in X$;
\item \textit{left-increasing},  if $x \unlhd y \text{ implies } F(x,z) \unlhd F(y,z)$, for any $x,y,z \in X$;
\item \textit{increasing}, if $x \unlhd y$ and $z \unlhd t \text{ implies }  F(x,z) \unlhd F(y,t)$, for any $x,y,z,t \in X$.
\end{enumerate} 
If a binary operation $F$ on $\mathbb{P}$ is increasing, then it is right- and left-increasing. The converse also holds if $\unlhd$ is transitive, {\em i.e.}, if $\mathbb{P}$ is a poset.

A binary operation $F$ on a trellis $\mathbb{T}=(X,\unlhd,\meet ,\join)$
is called:
\begin{enumerate}[label=(\roman*)]
\item \textit{conjunctive}, if $F(x,y)\unlhd x\meet y$, for any $x,y \in X$;
\item \textit{disjunctive}, if $x\join y \unlhd F(x,y)$, for any $x,y \in X$.
\end{enumerate}

\begin{remark}
Consider a trellis $\mathbb{T}=(X,\unlhd,\meet,\join)$.
\begin{enumerate}[label=(\roman*)]
\item The meet $\meet$ (resp.\ join $\join$) is conjunctive (resp.\  disjunctive). 
\item If a binary operation $F$ on $X$ satisfies
$F(x,y)\unlhd x$ and $F(x,y)\unlhd y$ 
(resp.\ $x \unlhd F(x,y)$ and $y \unlhd F(x,y)$), 
for any $x,y\in X$, then it is conjunctive (resp.\ disjunctive). 
The converse holds if $\unlhd$ is transitive ({\em i.e.}, if $\mathbb{T}$ is a lattice).
If the trellis $\mathbb{T}$ is bounded, and $F$ is right- and left-increasing and $1$ (resp.\ $0$) 
is the neutral element of $F$, then $F$ being conjunctive (resp.\ disjunctive) implies that 
$F(x,y)\unlhd x$ and $F(x,y)\unlhd y$ (resp.\ $x \unlhd F(x,y)$ and $y \unlhd F(x,y)$), for any $x,y\in X$. 
\end{enumerate}
\end{remark}

\begin{example}\label{ex:leftrightnotincreasing}
Consider the bounded trellis $\mathbb{T}=(\{0,a,b,c,1\}, \unlhd, \meet , \join)$ with 
the Hasse diagram shown in Fig.~\ref{Fig02} and the binary operation $F$ 
on $\mathbb{T}$ in Table~\ref{tablelni}.

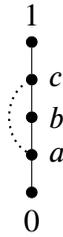
\begin{figure}[H]
\[\begin{tikzpicture}
\tikzstyle{estun}=[>=latex,thick,dotted]
    \vertex[fill] (0) at (0,0.5)  [label=below:$0$]  {};
    \vertex[fill] (a) at (0,1)  [label=right:$a$]  {};
    \vertex[fill] (b) at (0,1.5)  [label=right:$b$]  {};
    \vertex[fill] (c) at (0,2)  [label=right:$c$]  {};
    \vertex[fill] (1) at (0,2.5)  [label=above:$1$]  {};
   
    \path
        (0) edge (a)
        (a) edge (b)
        (b) edge (c)
        (c) edge (1)

        ;
   \draw[estun] (a) to [bend left=60] (c)
   ;
           
\end{tikzpicture}\]
\caption{Hasse diagram of the bounded trellis in Examples~\ref{ex:leftrightnotincreasing}~and~\ref{the meet and the join are not left- and right-increasing}.}\label{Fig02}
\end{figure}
\begin{table}[H]
\begin{center}
\begin{tabular}{|c|c|c|c|c|c|}
\hline 
$F$      & $0$ & $a$ & $b$ & $c$ & $1$ \\ 
\hline  
$0$	     & $0$& $0$ & $0$ & $0$ & $0$ \\  
\hline   
$a$	     & $0$& $a$ & $a$ & $b$ & $b$ \\ 
\hline 
$b$      & $0$& $a$ & $b$ & $b$ & $b$ \\ 
\hline 
$c$      & $0$& $b$ & $b$ & $c$ & $c$ \\
\hline 
$1$      & $0$& $b$ & $b$ & $c$ & $1$ \\ 
\hline 
\end{tabular}
\end{center}
      \caption{The binary operation of Example~\ref{ex:leftrightnotincreasing}.}
    \label{tablelni}
\end{table}
The operation $F$ is left-increasing, and due to its commutativity, also right-increa\-sing. However, 
it is not increasing since $a \lhd 1$ and $b \lhd c$, while $a=F(a,b) \ntrianglelefteq c=F(1,c)$. 
\end{example}

The meet and join operations of a trellis are not necessarily right- and left-increasing.

\begin{example}\label{the meet and the join are not left- and right-increasing}
Consider again the bounded trellis in Example~\ref{ex:leftrightnotincreasing}.
Since $b \unlhd c$ and $b \meet a=a \meet b= a \ntrianglelefteq 0= c \meet a=a \meet c$, the meet operation $\meet$ is neither right- nor left-increasing. A similar observation holds for the join operation $\join$.
\end{example}

In fact, the increasingness of $\meet$ and $\join$ is only satisfied in the lattice setting.

\begin{proposition}\label{meet-join_Increasing_equi_transitive}
Let $\mathbb{T}=(X,\unlhd ,\meet ,\join )$ be a trellis. The binary operation $\meet$ (or $\join$) is increasing if and only if $\unlhd$ is transitive ({\em i.e.}, $(X,\unlhd ,\meet ,\join )$ is a lattice). 
\end{proposition}

\begin{proof}
We give the proof for the meet operation $\meet$. Suppose that $\meet$ is increasing and $\unlhd$ is not transitive ({\em i.e.}, $ x\unlhd y \unlhd z$ and $x \ntrianglelefteq z$, for some $x,y,z \in X$). Since $\meet$  is increasing and $y\unlhd z$, it follows that $ x= x \meet y \unlhd x \meet z$. Hence, $x = x \meet z$. Thus, $ x \unlhd z$, a contradiction. For the converse, in a lattice the meet and join operations are increasing.
\end{proof}

Combining Theorem~\ref{Transitivity_equivalent_to_Meet_join_ass} and Proposition~\ref{meet-join_Increasing_equi_transitive} leads to the following corollary.

\begin{corollary}\label{Increasing_associative_means_transitivity}
Let $\mathbb{T}=(X,\unlhd,\meet,\join)$ be a trellis. The following statements are equivalent:
\begin{enumerate}[label=(\roman*),font=\upshape]  
\item $\meet$ (resp.\ $\join$) is increasing;
\item $\unlhd$ is transitive;
\item $\meet$ (resp.\ $\join$) is associative.
\end{enumerate}
\end{corollary}

\subsection{T-norms on a bounded psoset or trellis}
In this subsection, we introduce the notion of a t-norm on a bounded psoset and present some examples. 

\begin{definition}\label{t-norm}
Let $\mathbb{P}=(X,\unlhd,0,1)$ be a bounded psoset. A binary operation $T: X^2 \to X$ is called a triangular norm (t-norm, for short) on $\mathbb{P}$ if it is increasing, commutative, associative and has $1$ as neutral element, {\em i.e.}, $T(x,1)=x$, for any $x \in X$.
\end{definition}

\begin{remark}\noindent
\begin{enumerate}[label=(\roman*)]  
    \item The meet operation $\meet$ on a proper bounded trellis $\mathbb{T}=(X,\unlhd,\meet,\join,0,1)$ is not a t-norm on~$\mathbb{T}$ as it is not necessarily increasing or associative. 
     \item Any t-norm on a bounded trellis is conjunctive.
\end{enumerate}
\end{remark}

The pseudo-order $\unlhd$ can be extended pointwisely to a pseudo-order on t-norms, {\em i.e.},
$T_1\unlhd T_2$ if $T_1(x,y)\unlhd T_2(x,y)$ for any $x,y\in X$.

The following proposition shows that the existence of an idempotent t-norm on a given bounded trellis $\mathbb{T}$ implies that $\mathbb{T}$ is a bounded lattice. The proof is straightforward. 

\begin{proposition}\label{idempotent_t-norm_imply_lattice}
Let $\mathbb{T}=(X,\unlhd,\meet,\join,0,1)$ be a bounded trellis. If there exists an idempotent t-norm $T$ on~$\mathbb{T}$, then $T=\meet$ and $\mathbb{T}$ is a bounded lattice.
\end{proposition}

Note that this result can be stated more generally outside the study of t-norms as already any increasing, conjunctive and idempotent binary operation coincides with the meet operation. Hence, there exists no increasing, conjunctive and idempotent binary operation on a proper trellis.

Next, we present some examples. We start with the smallest possible t-norm.

\begin{example}\label{T_drastic}
Let $\mathbb{P}=(X,\unlhd,0,1)$ be a bounded psoset. The binary operation $T_\D$ on $X$ defined by:
\begin{equation}
T_\D(x,y)=
\left\{\begin{array}{ll}
{x} & \text {, if } y=1\\
{y} & \text {, if } x=1\\
{0} & \text {, otherwise}
 \end{array}\right.
\end{equation}
is the smallest t-norm on $\mathbb{P}$, and will be called the {\em drastic t-norm}.
\end{example}

Note that on a bounded trellis $\mathbb{T}=(X,\unlhd,\meet,\join,0,1)$,
the t-norm $T_\D$ can also be written as
\begin{equation}
T_\D(x,y)=
\left\{\begin{array}{ll}
{x \meet y} & \text{, if } x=1 \text{ or } y=1\\
        {0} & \text {, otherwise}
 \end{array}\right..
\end{equation}

On a bounded modular lattice, the binary operation $T_\Z$,
a slight modification of the t-norm $T_\D$, defined by:
\begin{equation}
T_\Z(x, y)=
\left\{\begin{array}{ll}
{x \meet y} & \text {, if } x \join y=1\\
 {0}        & \text {, otherwise}
 \end{array}\right.
\end{equation}
is also a t-norm~\cite{BaetsMesiar1999}. This result can be extended to the setting of bounded modular trellises under a suitable necessary and sufficient condition. The proof is based on the following proposition that recalls an interesting property of bounded modular trellises. This property will be also exploited in the next section.

\begin{proposition}{\rm \cite{Skala1971}}\label{Interesting_implication_modular_trellis}
Let $\mathbb{T}=(X,\unlhd,\meet,\join,0,1)$ be a bounded modular trellis. For any $x,y,z \in X$,
it holds that $x \unlhd z$ and $x \join y=1$ imply $x\meet y \unlhd z$.
\end{proposition}

\begin{proposition}\label{Special_t-norm}
Let $\mathbb{T}=(X,\unlhd,\meet,\join,0,1)$ be a bounded modular trellis. Then it holds that $T_\Z$ is a t-norm on~$\mathbb{T}$ if and only if 
\begin{equation}\label{condmod}
(\forall (x,y,z,t) \in X^4) (((x \meet y\neq 0) \mbox{ and } (x \join y =1)) 
\Rightarrow (x \join z)\join (y \join t)=1)\,. 
\end{equation}
\end{proposition}

\begin{proof}
Suppose that $\mathbb{T}=(X,\unlhd,\meet,\join,0,1)$ satisfies condition~\eqref{condmod}. Obviously, $T_\Z$ is commutative and has $1$ as neutral element.

Next, we prove that $T_\Z$ is increasing. Let $x,y,z,t \in X$ such that $ x \unlhd y$ and $ z \unlhd t$. We consider the following two cases: 
 \begin{enumerate}[label=(\roman*)]  
\item If $T_\Z(x,z)=0$, then it trivially holds that $T_\Z(x,z) \unlhd T_\Z(y,t)$.
\item If $T_\Z(x,z)=x\meet z\neq 0$, then $x \join z =1$. Due to
\eqref{condmod}, it holds that $(x \join y)\join (z \join t)=1$.
Since $x \unlhd y$ and $ z \unlhd t$, this implies $y\join t=1$, and hence 
$T_\Z(y,t)=y \meet t$.
Furthermore, from Proposition~\ref{Interesting_implication_modular_trellis} it follows that $x \meet z \unlhd y$ and $x \meet z \unlhd t$. Hence,  $x \meet z \unlhd y \meet t$. Thus, $T_\Z(x,z) \unlhd T_\Z(y, t)$. 
\end{enumerate}
Therefore, $T_\Z$ is increasing. 

Now, we prove that $T_\Z$ is associative. Let $ x ,y,z \in X$, then
\[
T_\Z(x,T_\Z(y,z))
=\left\{\begin{array}{ll}{x \meet (y \meet z)} & \text {, if } y \join z=1 \text { and } x \join(y \meet z)=1 \\ 
{0} & \text{, otherwise}\end{array}\right.
\]
and
\[
T_\Z(T_\Z(x,y),z)
=\left\{\begin{array}{ll}{(x\meet y)\meet z} & {\text {, if } x \join y=1 \text { and } z \join (x \meet y)=1} \\ 
{0} & {\text {, otherwise}}\end{array}\right.\,.
\]
First, we show that the condition ($y\join z=1$  and $x \join (y \meet z)=1$)  is equivalent to the condition ($x \join y=1$ and $z \join (x \meet y)=1$). Suppose that $y\join z=1$  and $x \join (y \meet z)=1$. Since $y \unlhd x \join y$ and $y \join z=1$, Proposition~\ref{Interesting_implication_modular_trellis} implies $ y \meet z \unlhd x \join y $. Together with $x\unlhd x\join y$, we get $1= x \join (y \meet z) \unlhd x \join y$. 
Thus, $x \join y=1$. Moreover, since $x \join(y \meet z)=1$, it holds that $y= y \meet(x \join (y \meet z))$. The modularity of $X$ implies 
$y=y \meet(x \join (y \meet z))=(x \meet y) \join(y \meet z)= ((x \meet y) \join z )\meet y$. Thus,  
$y \unlhd(x \meet y) \join z$. On the other hand, since 
$z \unlhd(x \meet y) \join z$, it follows that $y \join z \unlhd (x \meet y) \join z$. Hence, $(x \meet y) \join z=1$. The proof of the converse implication is similar.

Second, it now suffices to prove that $x \meet (y \meet z)= (x \meet y) \meet z$, for any $ x ,y,z \in X$ satisfying $y\join z=1$  and $x \join (y \meet z)=1$. Since $y \meet z \unlhd z$  and $x \join (y \meet z)=1$, Proposition~\ref{Interesting_implication_modular_trellis} implies $x \meet (y \meet z) \unlhd z$. Similarly, $y \meet z \unlhd y$ and $x \join (y \meet z)=1$ imply $x \meet (y \meet z) \unlhd y$. Together with $x \meet (y \meet z) \unlhd x$, we obtain $x \meet (y \meet z) \unlhd x \meet y$, and, hence, $x \meet (y \meet z) \unlhd (x \meet y) \meet z$. Similarly, we find $(x \meet y) \meet z \unlhd x \meet (y \meet z)$. Hence, $x \meet (y \meet z) = (x \meet y) \meet z$. Therefore, $T_\Z$ is associative. Consequently, $T_\Z$ is a t-norm on~$\mathbb{T}$. 

To prove the converse implication, suppose that $T_\Z$ is a t-norm on~$\mathbb{T}$. Let $x,y,z,t \in X$ be such that $x \meet y\neq 0$ and $x \join y =1$. It then holds that $T_\Z(x,y)= x \meet y$. Since $T_\Z$ is increasing, it follows that $T_\Z(x,y) \unlhd T_\Z(x \join z, y \join t)$. Hence, $ T_\Z(x \join z, y \join t) \not=0$. Thus, $(x \join z) \join  ( y \join t)=1$.
\end{proof}

\begin{remark}\label{Z_condition-not-satis}
Condition~\eqref{condmod} does not hold for any bounded modular trellis. Indeed, let $\mathbb{T}=(\{0,a,b,c,d,e,1\},\unlhd)$ be the bounded modular trellis given by the Hasse diagram in Fig.~\ref{Fig_Z-condmod}. One can verify that $\mathbb{T}$ does not satisfy condition~\eqref{condmod}.

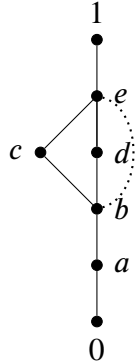
\begin{figure}[H]
\[\begin{tikzpicture}
\tikzstyle{estun}=[>=latex,thick,dotted]
    \vertex[fill] (0) at (0,0)  [label=below:$0$]  {};
    \vertex[fill] (a) at (0,0.75)  [label=right:$a$]  {};
    \vertex[fill] (b) at (0,1.5)  [label=right:$b$]  {};
    \vertex[fill] (c) at (-0.75,2.25)  [label=left:$c$]  {};
    \vertex[fill] (d) at (0,2.25)  [label=right:$d$]  {};
     \vertex[fill] (e) at (0,3)  [label=right:$e$]  {};
    \vertex[fill] (1) at (0,3.75)  [label=above:$1$]  {};
   
    \path
        (0) edge (a)
        (a) edge (b)
        (b) edge (c)
        (b) edge (d)
        (c) edge (e)
        (d) edge (e)
        (d) edge (1)
        
        ;
   \draw[estun] (b) to [bend right=75] (e)   
      
        ;
\end{tikzpicture}\]
\caption{Hasse diagram of the bounded modular trellis in Remark~\ref{Z_condition-not-satis}}\label{Fig_Z-condmod}
\end{figure}
\end{remark}

The following example presents a bounded modular trellis that satisfies condition~\eqref{condmod}.

\begin{example} \label{modular}
Let $\mathbb{T}=(\{0,a,b,c,d,e,f,1\}, \unlhd, \wedge, \vee)$ be the bounded modular trellis given by the Hasse diagram in Fig.~\ref{T_Z}. One easily verifies that condition~\eqref{condmod} is fulfilled.
 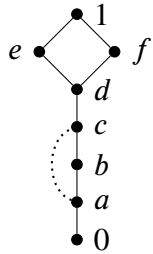
\begin{figure}[H]
\[\begin{tikzpicture}
\tikzstyle{estun}=[>=latex,thick,dotted]
    \vertex[fill] (0) at (0,0)  [label=right:$0$]  {};
    \vertex[fill] (a) at (0,0.5)  [label=right:$a$]  {};
    \vertex[fill] (b) at (0,1)  [label=right:$b$]  {};
    \vertex[fill] (c) at (0,1.5)  [label=right:$c$]  {};
    \vertex[fill] (d) at (0,2)  [label=right:$d$]  {};
    \vertex[fill] (e) at (-0.5,2.5)  [label=left:$e$]  {};
    \vertex[fill] (f) at (0.5,2.5)  [label=right:$f$]  {};
    \vertex[fill] (1) at (0,3)  [label=right:$1$]  {};
   
    \path
        (0) edge (a)
        (a) edge (b)
        (b) edge (c)
        (c) edge (d)
        (d) edge (e)
        (d) edge (f)   
        (f) edge (1)
        (e) edge (1)
        
       ;
   \draw[estun] (a) to [bend left=70] (c) ;                
\end{tikzpicture}\]
\caption{Hasse diagram of the bounded modular trellis in Example~\ref{modular}.}\label{T_Z}
\end{figure}
 The t-norm $T_\Z$ is listed in Table~\ref{mod}. To improve interpretability, from here on we indicate the points in which a t-norm coincides with the meet operation in gray.
 \begin{table}[H]
\begin{center}
\begin{tabular}{|c|c|c|c|c|c|c|c|c|}
\hline 
$T_\Z$ & $0$ & $a$ & $b$ & $c$ & $d$ & $e$& $f$& $1$ \\ 
\hline 
$0$ &  \cg $0$ & \cg $0$ & \cg $0$ & \cg $0$ & \cg $0$  &\cg  $0$ &\cg  $0$&\cg  $0$\\ 
\hline 
$a$ & \cg  $0$ &$0$ &$0$ &\cg  $0$ & $0$ & $0$ & $0$  &\cg  $a$ \\ 
\hline 
 $b$  & \cg $0$ & $0$ & $0$ & $0$ & $0$ & $0$  & $0$&\cg  $b$ \\ 
\hline 
$c$ &\cg $0$ & \cg $0$ & $0$ & $0$ & $0$ & $0$  & $0$&\cg  $c$ \\ 
\hline 
$d$ &\cg $0$ & $0$ & $0$ & $0$ & $0$ & $0$  & $0$&\cg  $d$ \\ 
\hline 
$e$ &\cg  $0$ &$0$ & $0$ & $0$ & $0$ & $0$  & \cg $d$&\cg  $e$ \\ 
\hline 
$f$&\cg  $0$ &$0$ & $0$ & $0$ & $0$ & \cg $d$  & $0$ &\cg  $f$ \\ 
\hline 
$1$& \cg $0$ &\cg  $a$ &\cg  $b$ &\cg  $c$ &\cg  $d$  &\cg  $e$&\cg  $f$  &\cg  $1$  \\ 
\hline 
\end{tabular}
\end{center}
\caption{The t-norm $T_\Z$ of Example~\ref{modular}.}\label{mod}
\end{table}
\end{example} 

The notion of a co-atom in a poset can be naturally extended to psosets.

\begin{definition}
Let $\mathbb{P}=(X,\unlhd,0,1)$ be a bounded psoset. An element $a \in X$ is called a \textit{co-atom} if it is a maximal element of the set $X \backslash\{1\}$.
\end{definition}

Note that any co-atom of a psoset is a right-transitive element. The following trivial proposition shows that any co-atom of a psoset can be used to define a t-norm on that psoset.

\begin{proposition}\label{T-norm_based_on_coatomes}
 Let $\mathbb{P}=(X,\unlhd,0,1)$ be a bounded psoset and
 $i$ be a co-atom of~$\mathbb{P}$. The binary operation $T_i$ on $X$ defined by 
 \begin{equation}
T_i(x,y)=
\left\{\begin{array}{ll}
{x} & {\text {, if } y=1}\\
{y} & {\text {, if } x=1} \\
{i}        & {\text {, if } (x,y)=(i,i)}\\
{0}        & {\text {, otherwise}}
\end{array}\right.
\end{equation}
is a t-norm on $\mathbb{P}$.
\end{proposition}

We conclude this subsection with some illustrative examples. Note that from here on, when presenting a t-norm on a finite bounded psoset or trellis in tabular form, in order to save space, 
we do not include the rows and columns corresponding to $0$ and $1$. Hence, we present a t-norm as a binary operation on $X\setminus\{0,1\}$.

\begin{example} \label{sixtnorms}
Consider again the bounded modular trellis in Example~\ref{ex:leftrightnotincreasing} satisfying condition~\eqref{condmod}. There are six t-norms on this bounded modular trellis, namely the drastic t-norm $T_1:=T_\D=T_\Z$, the t-norm $T_2:=T_c$ 
associated with the co-atom $c$ and the t-norms in Table~\ref{fourtnorms}.
\begin{table}[H]
\begin{center}
\begin{tabular}{|c|c|c|c|}
\hline 
$T_{3}$ & $a$& $b$ & $c$ \\ 
\hline 
$a$	   & $0$& $0$ & \cg $0$ \\  
\hline 
$b$	   & $0$& $0$ & $0$ \\ 
\hline 
$c$    & \cg $0$& $0$ & $b$ \\ 
\hline 
\end{tabular}
\quad
\begin{tabular}{|c|c|c|c|}
\hline 
$T_{4}$ & $a$& $b$ & $c$  \\ 
\hline 
$a$	   & $0$& $0$ & \cg $0$  \\  
\hline 
$b$	   & $0$& \cg $b$ & \cg $b$ \\ 
\hline 
$c$    & \cg $0$& \cg $b$ & $b$ \\ 
\hline 
\end{tabular}\\ 
\vspace*{.2cm}
\begin{tabular}{|c|c|c|c|}
\hline 
$T_{5}$ & $a$& $b$ & $c$ \\ 
\hline 
$a$	   & $0$& $0$ & \cg $0$ \\  
\hline 
$b$	   & $0$& $0$ & \cg $b$ \\ 
\hline 
$c$   & \cg $0$& \cg $b$ & \cg $c$ \\ 
\hline 
\end{tabular}
\quad
\begin{tabular}{|c|c|c|c|}
\hline 
$T_{6}$ & $a$& $b$ & $c$ \\ 
\hline 
$a$	  & $0$& $0$ & \cg $0$ \\  
\hline 
$b$	  & $0$& \cg $b$ & \cg $b$ \\ 
\hline 
$c$  & \cg $0$& \cg $b$ & \cg $c$ \\ 
\hline 
\end{tabular}
\caption{The t-norms $T_3$--$T_6$ of Example~\ref{sixtnorms}.}\label{fourtnorms}
\end{center}
\end{table}
One easily verifies that $T_{6}$ is the greatest t-norm on~$\mathbb{T}$. This set of six t-norms $\{T_{1}, \ldots, T_{6}\}$ on $\mathbb{T}$ constitutes a bounded lattice given by the Hasse diagram in Figure~\ref{figsixtnorms}.

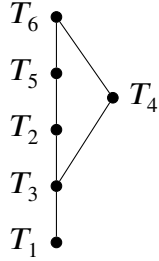
\begin{figure}[H]
\[\begin{tikzpicture}
\tikzstyle{estun}=[>=latex,thick,dotted]
    \vertex[fill] (T_1) at (0,0)  [label=left:$T_1$]  {};
    \vertex[fill] (T_3) at (0,0.75)  [label=left:$T_3$]  {};
    \vertex[fill] (T_2) at (0,1.5)  [label=left:$T_2$]  {};
    \vertex[fill] (T_5) at (0,2.25)  [label=left:$T_5$]  {};
    \vertex[fill] (T_4) at (0.75,1.925)  [label=right:$T_4$]  {};
    \vertex[fill] (T_6) at (0,3)  [label=left:$T_6$]  {};
   
    \path
        (T_1) edge (T_3)
        (T_3) edge (T_2)
        (T_2) edge (T_5)
        (T_3) edge (T_4)
        (T_5) edge (T_6)
        (T_4) edge (T_6)

       ;
\end{tikzpicture}\]
\caption{Hasse diagram of the bounded lattice of the six t-norms in Example~\ref{sixtnorms}.}\label{figsixtnorms}
\end{figure}

\end{example}

While the meet operation is the greatest t-norm on a bounded lattice, it is not a t-norm
on a proper bounded trellis. It is even possible that there is no greatest t-norm at all
on a proper bounded trellis.

\begin{example}\label{No_greatest_t-norm} 
Let $\mathbb{T}=(\{0,a,b,c,d,e,1\}, \unlhd , \meet , \join)$ be the bounded trellis given by the Hasse diagram in Fig.~\ref{Fig03}. 
\begin{figure}[H]
\[\begin{tikzpicture}
\tikzstyle{estun}=[>=latex,thick,dotted]
    \vertex[fill] (0) at (0,0)  [label=left:$0$]  {};
    \vertex[fill] (a) at (0,0.75)  [label=left:$a$]  {};
    \vertex[fill] (b) at (0,1.5)  [label=left:$b$]  {};
    \vertex[fill] (c) at (0,2.25)  [label=left:$c$]  {};
    \vertex[fill] (d) at (0.75,2.25)  [label=right:$d$]  {};
    \vertex[fill] (e) at (0,3)  [label=left:$e$]  {};
    \vertex[fill] (1) at (0,3.75)  [label=left:$1$]  {};
   
    \path
        (0) edge (a)
        (a) edge (b)
        (b) edge (c)
        (b) edge (d)
        (c) edge (e)
        (e) edge (1)
        (d) edge (1)

       ;
   \draw[estun] (a) to [bend left=70] (c) ;         
\end{tikzpicture}\]
\caption{Hasse diagram of the bounded trellis in Example~\ref{No_greatest_t-norm}.}\label{Fig03}
\end{figure}
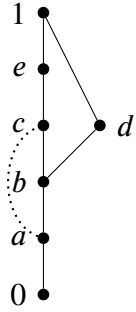

Consider the binary operations $T_{1}$ and $T_{2}$ on~$\mathbb{T}$ in Table~\ref{twotnorms}.
\begin{table}[H]
\begin{center}
\begin{tabular}{|c|c|c|c|c|c|}
\hline 
$T_{1}$ & $a$ & $b$ & $c$ & $d$ & $e$  \\ 
\hline 
$a$ &  $0$ & $0$ & \cg $0$ & \cg $a$ & $0$  \\ 
\hline 
 $b$  & $0$ & \cg $b$ & \cg $b$ & \cg $b$ & \cg $b$  \\ 
\hline 
$c$ & \cg $0$ & \cg $b$ & \cg $c$ & \cg $b$ & \cg $c$ \\ 
\hline 
$d$ & \cg $a$ & \cg $b$ & \cg $b$ & \cg $d$ & \cg $b$  \\ 
\hline 
$e$ & $0$ & \cg $b$ & \cg $c$ & \cg $b$ & \cg $e$  \\ 
\hline 
\end{tabular}
\quad
\begin{tabular}{|c|c|c|c|c|c|}
\hline 
$T_{2}$ & $a$ & $b$ & $c$ & $d$ & $e$\\ 
\hline 
$a$ &  $0$ & $0$ & \cg $0$ & $0$ & \cg $a$\\ 
\hline 
 $b$ & $0$ & \cg $b$ & \cg $b$ & \cg $b$ & \cg $b$\\ 
\hline 
$c$ & \cg $0$ & \cg $b$ & \cg $c$ & \cg $b$ & \cg $c$\\ 
\hline 
$d$ & $0$ & \cg $b$ & \cg $b$ & \cg $d$ & \cg $b$\\ 
\hline 
$e$ &  \cg $a$ & \cg $b$ & \cg $c$ & \cg $b$ & \cg $e$\\ 
\hline 
\end{tabular}
\caption{The two t-norms of Example~\ref{No_greatest_t-norm}.}\label{twotnorms}
\end{center}
\end{table}
These binary operations are maximal t-norms on~$\mathbb{T}$. Indeed, one easily verifies that $T_1$ and $T_2$ are t-norms on~$\mathbb{T}$.
Next, we only prove that $T_1$ is maximal, as the proof for $T_2$ is similar. Let $T'$ be an arbitrary t-norm on~$\mathbb{T}$ such that $T_{1}\unlhd T^{'}$. The proof goes by cases:
\begin{enumerate}[label=(\roman*)]  
\item If ($x,y \in \{0,b,c,d,e,1\}$), ($x=a$ and $y \in \{0,c,d,1\}$) or ($x \in \{0,c,d,1\}$
and $y=a$), then  $T_{1}(x,y)= x \wedge y$. Hence, $T'(x,y)=T_{1}(x,y)$.
\item If ($x=a$ and $y \in \{a,b\}$) or ($x \in \{a,b\}$ and $y=a$), then it follows from the increasingness of $T'$ that $T'(a,a) \unlhd  T'(a,b) \unlhd T'(a,c) \unlhd a \wedge c =0$. Thus, $T'(a,a) =  T'(a,b) =0$. Hence, $T'(x,y)=T_{1}(x,y)$. 
\item If $(x,y)=(a,e)$ or $(x,y)=(e,a)$, then two possible subcases can be considered.
\begin{enumerate}[label=(\roman*)]  
\item[(a)] If $T'(a,e)=0$, then $T'(x,y)=T_{1}(x,y)$. 
\item[(b)] If $T'(a,e)= a$, then the fact that $T'(a,T'(e,d))=T'(a,b)=0$ and $T'(T'(a,e),d)=T'(a,d)=a$,
which contradicts the associativity of $T'$. Hence, this case is impossible.
\end{enumerate}
\end{enumerate}
Thus, $T_{1}$ is a maximal t-norm on~$\mathbb{T}$.
\end{example} 

\section{A generic construction method based on interior operators}\label{Int_based-construction_of_t-norms}
As mentioned several times before, on a proper bounded trellis the meet operation is not a t-norm. 
In this section, inspired by the t-norms $T_{\mathrm D}$ and $T_\mathrm{Z}$, we investigate whether 
we can build a t-norm starting from the meet operation on an interior range of a given $\meet$-semi-trellis. 
First, we extend the notion of an interior operator~\cite{Blyth2005} to the trellis setting. 

\begin{definition}\label{Def:interior-operator}
 Let $\mathbb{T}=(X,\unlhd,\meet)$ be a $\meet$-semi-trellis. A mapping $I \colon X \to X$
 is called an interior operator on $\mathbb{T}$ if it satisfies the following three conditions: 
\begin{enumerate}[label=(\roman*)] 
\item $I(x) \unlhd x$, for any $x\in X$;
\item $I$ is idempotent, {\em i.e.}, $I(I(x)) = I(x)$, for any $x\in X$;
\item $I$ is a meet-homomorphism, {\em i.e.}, $I(x \meet y) = I(x) \meet I(y)$, for any $x,y\in X$.
\end{enumerate}
\end{definition}

\begin{definition}
Let $\mathbb{T}=(X,\unlhd,\meet)$ be a $\meet$-semi-trellis. A subset $R$ of $X$ is called
an interior range of $\mathbb{T}$ if there exists an interior operator $I$ on $\mathbb{T}$ such that
$R=I(X)=\{I(x)\mid x\in X\}$.
\end{definition}

The following proposition is immediate.

\begin{proposition}\label{additional_properties_int}
  Let $\mathbb{T}=(X,\unlhd,\meet)$ be a $\meet$-semi-trellis and $I$ an interior operator on $\mathbb{T}$ with range $R_I$. It holds that: 
 \begin{enumerate}[label=(\roman*),font=\upshape] 
\item if $x \in R_I$, then $I(x)=x$;
\item $I$ is increasing, {\em i.e.}, if $x\unlhd y$, then $I(x) \unlhd I(y)$, for any $x,y\in X$. 
\end{enumerate}
\end{proposition}

The following proposition discusses the structure of the range of an interior operator on a given $\meet$-semi-trellis. 

\begin{proposition}\label{structure_of_R_I}
 Let $\mathbb{T}=(X,\unlhd,\meet)$ be a $\meet$-semi-trellis and $I$ an interior operator on $\mathbb{T}$ with range $R_I$.
 It holds that:
\begin{enumerate}[label=(\roman*),font=\upshape] 
\item $(R_I,\unlhd,\meet)$ is a $\meet$-sub-trellis of $\mathbb{T}$;
\item if $\mathbb{T}=(X,\unlhd,\meet,0,1)$ is bounded, then $(R_I,\unlhd,\meet,0, I(1))$ is also bounded.  
\end{enumerate}  
\end{proposition}

\begin{proof}\noindent
\begin{enumerate}[label=(\roman*)] 
\item Let $x,y \in R_I$. Proposition~\ref{additional_properties_int}(i) and Definition~\ref{Def:interior-operator}(iii) guarantee that $x\meet y=I(x)\meet I(y)=I(x\meet y)$. Hence, $x\meet y\in R_I$. Thus, $(R_I,\unlhd,\meet)$ is a $\meet$-sub-trellis of $\mathbb{T}$.
\item Follows from Proposition~\ref{additional_properties_int}.
\end{enumerate} 
\end{proof}

In case of a bounded $\meet$-semi-trellis, some authors require an interior 
operator to satisfy $I(1)=1$~\cite{DvorakHolcapek2020}.

The following result follows from Proposition~\ref{structure_of_R_I} and Corollary~\ref{subtrellis_implies_sublattice}.
\begin{proposition}\label{Lattice_structure_of_R_I}
 Let $\mathbb{T}=(X,\unlhd,\meet,0,1)$ be a bounded $\meet$-semi-trellis and $I$ an interior operator on $\mathbb{T}$ with range $R_I$. If $R_{I}\subseteq X^\rtr$ or $R_{I}\subseteq X^\ltr$, then $(R_I,\unlhd,\meet,0, I(1))$ is a bounded $\meet$-sub-lattice of $\mathbb{T}$.
\end{proposition}

For a given t-norm $V$ on a bounded $\meet$-sub-trellis $(R_I,\unlhd,\meet,0,I(1))$, we define the following binary operation on~$\mathbb{T}$:
\begin{equation}\label{T^{I,V}}
T^{I,V}(x,y)=
\left\{\begin{array}{ll}
{y}                   & \text {, if } x=1 \\
{x}                   & \text {, if } y=1 \\
{V(I(x),I(y))} & \text {, otherwise.}
 \end{array}\right.
\end{equation}

In view of Proposition~\ref{additional_properties_int}(i), we can express 
$T^{I,V}$ as:
\begin{equation}\label{T^{I,V_details}}
T^{I,V}(x,y)=
\left\{\begin{array}{ll}
{y}                   & \text {, if } x=1 \\
{x}                   & \text {, if } y=1 \\
{V(x,y)} & \text {, if } x,y\in R_I \\
{V(I(x),I(y))}  & \text {, otherwise.}
 \end{array}\right.
\end{equation}

The following theorem shows that $T^{I,V}$ is a t-norm on~$\mathbb{T}$ when $R_I$ is a subset of~$X^\rtr$.

\begin{theorem}\label{Extended_t-norm_based_on_V_and_int-operator}
 Let $\mathbb{T}=(X,\unlhd,\meet,0,1)$ be a bounded $\meet$-semi-trellis, $I$ an interior operator on $\mathbb{T}$ with range $R_I$. If $R_I$ is a subset of $X^\rtr$ and $V$ is a t-norm on the bounded $\meet$-sub-lattice 
 $(R_I,\unlhd,\meet,0,I(1))$, then the binary operation $T^{I,V}$ is a t-norm on~$\mathbb{T}$. 
\end{theorem}

\begin{proof}
Obviously, $T^{I,V}$ is commutative and has $1$ as neutral element. 

Next, we prove that $T^{I,V}$ is increasing. Let  $x,y,z,t \in X $ such that $x \unlhd y$ and $z \unlhd t$. 
We consider the following two cases:
\begin{enumerate}[label=(\roman*)] 
\item First case: $x=1$, and hence $y=1$.
Then $T^{I,V}(x,z) = z$ and $T^{I,V}(y,t)=t$. Hence, $T^{I,V}(x,z) \unlhd T^{I,V}(y,t)$.
The same holds when $z=1$.
\item Second case: $x\neq 1$ and $z\neq 1$. Two subcases can be distinguished: ($y=1$ or $t=1$) or ($y\neq 1$ and $t\neq 1$). 
\begin{enumerate}[label=(\roman*)]  
\item[(a)]  First subcase: $y=1$. Since $T^{I,V}(x,z) = V(I(x),I(z))$, it follows from the increasingness of $I$ on $\mathbb{T}$ (Proposition~\ref{additional_properties_int}(ii))  and the increasingness of $V$ on $R_I$ that 
$T^{I,V}(x,z) = V(I(x),I(z))\unlhd  V(I(y),I(t))$. If $y=1$, then $T^{I,V}(x,z) \unlhd V(I(1),I(t))=I(t)\unlhd t= T^{I,V}(1,t)=T^{I,V}(y,t)$. The fact that $T^{I,V}(x,z)=V(I(x),I(z))\in R_I\subseteq X^\rtr$ implies that $T^{I,V}(x,z) \unlhd T^{I,V}(y,t)$. The proof is similar for $t=1$. 
\item[(b)] Second subcase: $y\neq 1$ and $t\neq 1$. Since $I$ is increasing on $\mathbb{T}$ and $V$ is increasing on $R_I$, it follows that $T^{I,V}(x,z) = V(I(x),I(z))\unlhd  V(I(y),I(t))= T^{I,V}(y,t)$. 
\end{enumerate}
\end{enumerate}

Therefore, $T^{I,V}$ is increasing. 

Next, we prove that $T^{I,V}$ is associative. Let $x,y,z \in X$, then we need to prove that 
\[T^{I,V}(x,T^{I,V}(y,z))= T^{I,V}(T^{I,V}(x,y),z)\,.\]
We consider the following two cases: 
 \begin{enumerate}[label=(\roman*)]  
\item If $1\in\{x,y,z\}$, then the equality trivially holds.
\item If $1\notin \{x,y,z\}$, then 
\[T^{I,V}(x,T^{I,V}(y,z))= V(I(x),V(I(y),I(z))\,.\]

The associativity of $V$ on $R_I$ guarantees that 
\[V(I(x),V(I(y),I(z)))=V(V(I(x),I(y)),I(z))\,.\]
Hence, 
\[T^{I,V}(x,T^{I,V}(y,z))= T^{I,V}(T^{I,V}(x,y),z)\,.\]
\end{enumerate}
Thus, $T^{I,V}$ is associative. Therefore, $T^{I,V}$ is a t-norm on~$\mathbb{T}$.
\end{proof}

Considering as t-norm $V$ the meet operation restricted to $R_I$ leads to the following corollary.

\begin{corollary}\label{t-norm_based_on_int-operator}
 Let $\mathbb{T}=(X,\unlhd,\meet,0,1)$ be a bounded $\meet$-semi-trellis and $I$ an interior operator on $\mathbb{T}$ with range $R_I$.  If $R_I$ is a subset of $X^\rtr$, then the binary operation $T^{I}$ on $\mathbb{T}$ defined as:  
\begin{equation}\label{T^{I_details}}
T^{I}(x,y)=
\left\{\begin{array}{ll}
{y}                   & \text {, if } x=1 \\
{x}                   & \text {, if } y=1 \\
{x\meet y} & \text {, if } x,y\in R_I \\
{I(x)\meet I(y)} & \text {, otherwise.}
 \end{array}\right.
\end{equation}
is a t-norm on~$\mathbb{T}$.
\end{corollary}

The following proposition shows that $T^{I}$ is a meet-preserving t-norm, {\em i.e.}, its partial mappings are meet-homomorphisms. 

\begin{proposition}\label{T_int_is_meet_preserving}
Let $\mathbb{T}=(X,\unlhd,\meet,0,1)$ be a bounded $\meet$-semi-trellis
and $I$ an interior operator on $\mathbb{T}$ with range $R_I\subseteq X^\rtr$. Then the t-norm $T^{I}$ is meet-preserving, {\em i.e.}, 
\[T^{I}(x,y\meet z)=T^{I}(x,y)\meet T^{I}(x,z)\,,
\]
for any $x,y,z\in X$.
\end{proposition}

\begin{proof}
The proof follows from the fact that $I$ is a meet-homomorphism and the associativity of $T^{I}$.
\end{proof}

\section{A particular family of interior operators}\label{T-norms_based_on_lambda_A}
In this section, we consider a particular family of interior operators on a bounded trellis.
For a subset $A$ of a given bounded trellis $\mathbb{T}=(X,\unlhd,\meet,\join,0,1)$, we define the mapping $\lambda_A \colon X \to X$:
\begin{equation}\label{lambda^alpha}
\lambda_A(x)=\bigvee \{ a \in A \mid  a \unlhd x \}=\bigvee(A\,\cap\downarrow x)\,.
\end{equation}
In general, this mapping is not well defined since the supremum $\bigvee(A\,\cap\downarrow x)$ does not necessarily exist. 
However, if $A$ is a finite subset of $X^\rtr$ containing $0$, then Proposition~\ref{Closeness_of_join_for_subtrellis}
guarantees that it is well defined. 

The following proposition shows that $\lambda_A$ is an interior operator. 

\begin{proposition}\label{lambda_is_an_interior_operator}
 Let $\mathbb{T}=(X,\unlhd,\meet,\join,0,1)$ be a bounded trellis and $A$ a finite subset of $X^\rtr$ containing $0$. If $A$ is a sub-trellis of~$\mathbb{T}$, then $\lambda_A$ is an interior operator on~$\mathbb{T}$ and $R_{\lambda_A}=A$.
\end{proposition}

 \begin{proof}\noindent
 \begin{enumerate}[label=(\roman*)] 
\item Obviously, $\lambda_A(x) \unlhd x$, for any $x\in X$.
\item Let $x \in X$. From Proposition~\ref{Closeness_of_join_for_subtrellis}, it follows that $\lambda_A(x) \in A$. Hence, $\lambda_A(\lambda_A(x))= \lambda_A (x)$.  Thus, $\lambda_A$ is idempotent.
\item Let $x,y \in X$, then $\lambda_A(x),\lambda_A(y)\in A$ and  $\lambda_A(x) \meet \lambda_A(y)\in A$. Since $A\subseteq X^\rtr$, $\lambda_A(x) \meet \lambda_A
(y)\unlhd \lambda_A(x) \unlhd  x$ implies $\lambda_A(x) \meet \lambda_A(y) \unlhd  x$. Similarly, $\lambda_A(x) \meet \lambda_A(y) \unlhd  y$. Hence, $\lambda_A
(x) \meet \lambda_A(y) \unlhd  x\meet y$. The fact that $\lambda_A
$ is idempotent and increasing implies that  $\lambda_A
(x) \meet \lambda_A(y)=\lambda_A
(\lambda_A(x) \meet \lambda_A(y)) \unlhd \lambda_A (x\meet y)$. The increasingness of $\lambda_A$ also implies that $\lambda_A
(x\meet y)  \unlhd \lambda_A
(x) \meet \lambda_A
(y)$. Hence, $\lambda_A
(x \meet y)=\lambda_A
(x) \meet \lambda_A
(y)$. Thus, $\lambda_A
$ is a meet-homomorphism.
\end{enumerate}
Obviously, $\lambda_A(x)=x$, for any $x\in A$. Hence, $R_{\lambda_A}=A$.
\end{proof}

Note that due to Proposition~\ref{Lattice_structure_of_R_I},
under the conditions of Proposition~\ref{lambda_is_an_interior_operator}, the sub-trellis $A$ is a bounded $\meet$-sub-lattice.

Combining Theorem~\ref{Extended_t-norm_based_on_V_and_int-operator} and Proposition~\ref{lambda_is_an_interior_operator}, we obtain the following result showing that $T^{\lambda_A,V}$ is a t-norm on~$\mathbb{T}$ when $A$ is a sub-trellis of~$\mathbb{T}$. Note that we use the notation $T^{[A,V]}$ instead of $T^{\lambda_A,V}$.

\begin{corollary}\label{T_A,V_t-norm}
 Let $\mathbb{T}=(X,\unlhd,\meet,\join,0,1)$ be a bounded trellis and $A$ a finite subset of $X^\rtr$ containing $0$. If $A$ 
 is a sub-trellis of~$\mathbb{T}$ and $V$ is a t-norm on the bounded $\meet$-sub-lattice $(A,\unlhd,\meet,\join,0,\lambda_A(1))$, 
 then the binary operation $T^{[A,V]}$ on $\mathbb{T}$ defined as:
 \begin{equation}\label{T^{A,V}_in_detail}
T^{[A,V]}(x,y)=
\left\{\begin{array}{ll}
{y}                               & \text {, if } x=1 \\
{x}                               & \text {, if } y=1 \\
{V(x,y)}                       & \text {, if } x,y\in A \\
{V(\lambda_A(x) ,\lambda_A(y))} & \text {, otherwise.}
 \end{array}\right.
\end{equation}
is a t-norm on~$\mathbb{T}$.
 \end{corollary}

In the same line, combining Corollary~\ref{t-norm_based_on_int-operator} and Proposition~\ref{lambda_is_an_interior_operator}, we obtain the following result showing that $T^{\lambda_A}$ is a t-norm on~$\mathbb{T}$ when $A$ is a sub-trellis of~$\mathbb{T}$. Note that 
we use the notation $T^{[A]}$ instead of $T^{\lambda_A}$. 


\begin{corollary}\label{T_A_t-norm}
 Let $\mathbb{T}=(X,\unlhd,\meet,\join,0,1)$ be a bounded trellis and $A$ a finite subset of $X^\rtr$ containing $0$. If $A$ is a sub-trellis of~$\mathbb{T}$, then the binary operation $T^{[A]}$ on~$\mathbb{T}$ defined as: 
  \begin{equation}\label{T^{A}_in_detail}
T^{[A]}(x,y)=
\left\{\begin{array}{ll}
{y}                               & \text {, if } x=1 \\
{x}                               & \text {, if } y=1 \\
{x \meet y}                       & \text {, if } x,y\in A \\
{\lambda_A(x) \meet \lambda_A(y)} & \text {, otherwise.}
 \end{array}\right.
\end{equation}
is a t-norm on~$\mathbb{T}$.
 \end{corollary}

\begin{remark}\label{co-atoms_subtrellis_and_T=meet}
We list some basic observations:
\begin{enumerate}[label=(\roman*),font=\upshape] 
\item If $A=\{0,i\}$, with $i$ a co-atom of~$\mathbb{T}$, then $T^{[A]}=T_i$.
\item If $A=\{0,1\}$, then $T^{[A]}=T_\D$.
\item If $\mathbb{T}=(X,\unlhd,\meet,\join,0,1)$ is a bounded lattice, then $T^{[X^\alpha]}=\meet$, for any $\alpha \in\{\ass,\meetAss,\joinAss,\tr,\rtr\}$.
\end{enumerate}
\end{remark}

Using Propositions~\ref{Specific_sets_inclusion} and~\ref{X^tr=X^ass_X^dis_sublattice}, and Corollary~\ref{T_A_t-norm}, we obtain the following corollaries.

\begin{corollary}\label{T_alpha_t-norms}
 Let $\mathbb{T}=(X,\unlhd,\meet,\join,0,1)$ be a bounded trellis and~$\alpha\in\{\dis,\ass,\meetAss,$ $\joinAss,\tr,\rtr\}$. If $X^\alpha$ is a finite sub-trellis of~$\mathbb{T}$, then the binary operation $T^{[X^\alpha]}$ is a t-norm on~$\mathbb{T}$.
 \end{corollary}

Next, we illustrate the above results. 

\begin{example}\label{illu}
Let $\mathbb{T}=(\lbrace 0,a,b,c,d,e,1 \rbrace,\unlhd)$ be the bounded trellis given by the Hasse diagram in Fig.~\ref{Fig04}. 
 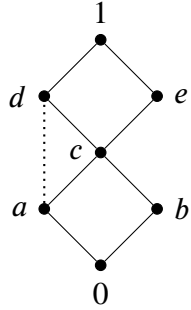
\begin{figure}[H]
\[\begin{tikzpicture}
\tikzstyle{estun}=[>=latex,thick,dotted]
    \vertex[fill] (0) at (0,0)  [label=below:$0$]  {};
    \vertex[fill] (a) at (-0.75,0.75)  [label=left:$a$]  {};
    \vertex[fill] (b) at (0.75,0.75)  [label=right:$b$]  {};
    \vertex[fill] (c) at (0,1.5)  [label=left:$c$]  {};
    \vertex[fill] (d) at (-0.75,2.25)  [label=left:$d$]  {};
    \vertex[fill] (e) at (0.75,2.25)  [label=right:$e$]  {};
    \vertex[fill] (1) at (0,3)  [label=above:$1$]  {};
   
    \path
        (0) edge (a)
        (0) edge (b)
        (b) edge (c)
        (a) edge (c)
        (c) edge (e)
        (c) edge (d)
        (d) edge (1)
        (e) edge (1) 
       
       ;
   \draw[estun] (a) to (d);   
        
\end{tikzpicture}\]
\caption{Hasse diagram of the bounded trellis in Example~\ref{illu}.}\label{Fig04}
\end{figure}

One easily verifies that $X^{\rtr} = \{0,b,c,d,e,1\}$ is a sub-lattice of~$\mathbb{T}$. The interior operator $\lambda_{X^{\rtr}}$ is listed in Table~\ref{IntOp_lammbda_A}. 
\begin{table}[H]
\begin{center}
\begin{tabular}{|c|c|c|c|c|c|c|c|}
\hline 
      & $0$ & $a$ & $b$ & $c$ & $d$ & $e$ & $1$ \\ 
\hline 
$\lambda_{X^{\rtr}}$     & $0$ & $0$ & $b$ & $c$ & $d$ & $e$  & $1$  \\ 
\hline
\end{tabular}
\caption{The interior operator $\lambda_{X^{\rtr}}$ of Example~\ref{illu}.}\label{IntOp_lammbda_A}
\end{center}
\end{table}
It is obvious that if $A$ is a bounded sub-lattice of~$\mathbb{T}$, then for any $a\in A$, the binary operation $V_a$ on $A$ defined by $V_a(x,y)=x\meet y\meet a$ is a t-norm on $A$. Next, we illustrate Corollary~\ref{T_A,V_t-norm}. The t-norms $T^{[X^\rtr,V_b]}$, $T^{[X^\rtr,V_c]}$, $T^{[X^\rtr,V_d]}$  and $T^{[X^\rtr,V_e]}$ are listed in Table~\ref{A,V_b}.
\begin{table}[H]
\begin{center}
\begin{tabular}{|c|c|c|c|c|c|}
\hline 
$T^{[X^\rtr,V_b]}$ & $a$ & $b$ & $c$ & $d$ & $e$ \\ 
\hline 
$a$ &    $0$ & \cg $0$ & $0$ & \cg $0$ & $0$ \\ 
\hline 
 $b$ &	 \cg $0$ & \cg $b$ & \cg $b$ & \cg $b$ & \cg $b$\\ 
\hline 
$c$ &	 $0$ & \cg $b$ & $b$ & $b$ & $b$\\ 
\hline 
$d$ &    \cg $0$ & \cg $b$ & $b$ & $b$ & $b$\\ 
\hline 
$e$ &    $0$ & \cg $b$ & $b$ & $b$ & $b$\\ 
\hline 
\end{tabular}
\quad 
\begin{tabular}{|c|c|c|c|c|c|}
\hline
$T^{[X^\rtr,V_c]}$ & $a$ & $b$ & $c$ & $d$ & $e$ \\
\hline 
$a$ &  $0$ & \cg $0$ & $0$ & \cg $0$ & $0$ \\ 
\hline 
 $b$ & \cg $0$ & \cg $b$ & \cg $b$ & \cg $b$ & \cg $b$ \\ \hline 
$c$ & $0$ & \cg $b$ & \cg $c$ & \cg $c$ & \cg $c$  \\ 
\hline 
$d$ &  \cg $0$ & \cg $b$ & \cg $c$ & $c$ & \cg $c$ \\ \hline 
$e$ &   $0$ & \cg $b$ & \cg $c$ & \cg $c$ & $c$ \\ 
\hline 
\end{tabular}\\ 
\vspace*{.2cm}
\begin{tabular}{|c|c|c|c|c|c|}
\hline
$T^{[X^\rtr,V_d]}$ & $a$ & $b$ & $c$ & $d$ & $e$ \\
\hline 
$a$ &  $0$ & \cg $0$ & $0$ & \cg $0$ & $0$ \\ 
\hline 
 $b$ & \cg $0$ & \cg $b$ & \cg $b$ & \cg $b$ & \cg $b$ \\ \hline 
$c$ & $0$ & \cg $b$ & \cg $c$ & \cg $c$ & \cg $c$  \\ 
\hline 
$d$ &  \cg $0$ & \cg $b$ & \cg $c$ & \cg$d$ & \cg $c$ \\ \hline 
$e$ &   $0$ & \cg $b$ & \cg $c$ & \cg $c$ & $c$ \\ 
\hline 
\end{tabular} 
\quad 
\begin{tabular}{|c|c|c|c|c|c|}
\hline
$T^{[X^\rtr,V_e]}$ & $a$ & $b$ & $c$ & $d$ & $e$ \\
\hline 
$a$ &  $0$ & \cg $0$ & $0$ & \cg $0$ & $0$ \\ 
\hline 
 $b$ & \cg $0$ & \cg $b$ & \cg $b$ & \cg $b$ & \cg $b$ \\ \hline 
$c$ & $0$ & \cg $b$ & \cg $c$ & \cg $c$ & \cg $c$  \\ 
\hline 
$d$ &  \cg $0$ & \cg $b$ & \cg $c$ &  $c$ & \cg $c$ \\ \hline 
$e$ &   $0$ & \cg $b$ & \cg $c$ & \cg $c$ & \cg $e$ \\ 
\hline 
\end{tabular}
\caption{The t-norms $T^{[X^\rtr,V_b]}$, $T^{[X^\rtr,V_c]}$, $T^{[X^\rtr,V_d]}$ and $T^{[X^\rtr,V_e]}$ of Example~\ref{illu}.}\label{A,V_b}
\end{center}
\end{table}
The greatest t-norm on~$\mathbb{T}$ is given by Table~\ref{greatest_t-norm}, and differs from the meet operation in some points involving the non-right-transitive element $a$. 
\begin{table}[H]
\begin{center}
\begin{tabular}{|c|c|c|c|c|c|}
\hline 
$T$  &       $a$ & $b$ & $c$ & $d$ & $e$ \\ 
\hline 
$a$ &        $0$ & \cg $0$ & $0$ & \cg $0$ & \cg $a$ \\ 
\hline 
 $b$ &	     \cg $0$ & \cg $b$ & \cg $b$ & \cg $b$ & \cg $b$ \\ 
\hline 
$c$ &	    $0$ & \cg $b$ & \cg $c$ & \cg $c$ & \cg $c$ \\ 
\hline 
$d$ &       \cg $0$ & \cg $b$ & \cg $c$ & \cg $d$ & \cg $c$ \\ 
\hline 
$e$ &       \cg $a$ & \cg $b$ & \cg $c$ & \cg $c$ & \cg $e$ \\ 
\hline 
\end{tabular}
\caption{The greatest t-norm of Example~\ref{illu}.}\label{greatest_t-norm}
\end{center}
\end{table}
\end{example}

\begin{example}\label{Detail_Example3.5}
Consider again the bounded modular trellis in Example~\ref{sixtnorms}
and the six t-norms on it. One easily verifies that $X^\rtr = \{0,b,c,1\}$ and $X^\dis=X^\ass=X^\meetAss=X^\joinAss=X^\tr=\{ 0,1\}$.
It is clear that $T_1=T^{[X^\dis]}=T^{[X^\ass]}=T^{[X^\meetAss]}=T^{[X^\joinAss]}=T^{[X^\tr]}(=T_\D)$, $T_2=T^{[\{0,c\}]}=T^{[\{0,c,1\}]}$, $T_4=T^{[\{0,b\}]}=T^{[\{0,b,1\}]}$ and $T_6=T^{[\{0,b,c\}]}=T^{[X^\rtr]}$. Since the subsets $A$ considered cover all subtrellises of $X^\rtr$ containing $0$, it follows that the t-norms $T_3$ and $T_5$ cannot be obtained using the construction method in Corollary~\ref{T_A_t-norm}.
\end{example}

The following example illustrates Corollary~\ref{T_A_t-norm} on a bounded trellis that includes a cycle. 

\begin{example}\label{t-norm_trellis_with_cycle} 
Let $\mathbb{T}=(\{0,a,b,c,d,e,f,1\}, \unlhd , \meet , \join)$ be the bounded trellis given by the Hasse diagram in Fig.~\ref{Fig07} and pseudo-order $\unlhd$
in Table~\ref{tablepseudo} (not mentioning $0$ and $1$). Note that $\{b,c,e,f\}$ is a cycle. 

\begin{figure}[H]
\[\begin{tikzpicture}
\tikzstyle{estun}=[>=latex,thick,dotted]
    \vertex[fill] (0) at (0,0)  [label=right:$0$]  {};
    \vertex[fill] (a) at (0,0.75)  [label=right:$a$]  {};
    \vertex[fill] (b) at (0,1.5)  [label=right:$b$]  {};
    \vertex[fill] (c) at (0,2.25)  [label=left:$c$]  {};
    \vertex[fill] (d) at (1,3)  [label=right:$d$]  {};
    \vertex[fill] (e) at (0,3)  [label=left:$e$]  {};
    \vertex[fill] (f) at (0,3.75)  [label=left:$f$]  {};
    \vertex[fill] (1) at (0,4.50)  [label=right:$1$]  {};
   
    \path
        (0) edge (a)
        (a) edge (b)
        (b) edge (d)
        (f) edge (1)
        (d) edge (1)   
            ;
        \draw [very thick] (b) to (c);
        \draw [very thick] (c) to (e);
        \draw [very thick] (e) to (f);

  \draw[estun] (b) to [bend left=65] (e); 
   \draw[estun] (b) to [bend right=30] (f);
  \draw[very thick, ->-=0.6] (f) to [bend right=70] (b);
\end{tikzpicture}\]
\caption{Hasse diagram of the bounded trellis in Example~\ref{t-norm_trellis_with_cycle}.}\label{Fig07}
\end{figure}
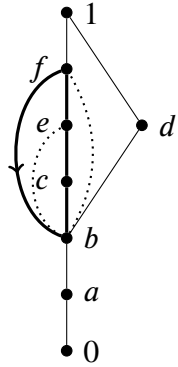
\begin{table}[H]
 \begin{center}
 \begin{tabular}{|c|c|c|c|c|c|c|}
\hline 
$\unlhd$ & $a$ & $b$ & $c$ & $d$ & $e$ & $f$ \\
\hline   
$a$	     &$1$ & $1$& $1$ & $1$ & $1$ & $1$ \\ 
\hline 
$b$	     &$0$ & $1$& $1$ & $1$ & $0$ & $0$ \\ 
\hline 
$c$      &$0$ & $0$& $1$ & $0$ & $1$ & $1$ \\ 
\hline 
$d$      &$0$ & $0$& $0$ & $1$ & $0$ & $0$ \\
\hline 
$e$      &$0$ & $0$& $0$ & $0$ & $1$ & $1$ \\ 
\hline 
$f$      &$0$ & $1$& $0$ & $0$ & $0$ & $1$ \\ 
\hline 
\end{tabular}
    \caption{Pseudo-order of the bounded trellis in Example~\ref{t-norm_trellis_with_cycle}.}
    \label{tablepseudo}
    \end{center}
\end{table}

One easily verifies that $X^{\rtr} = \{0,a,d,1\}$ is a sub-trellis of~$\mathbb{T}$. The interior operator $\lambda_{X^{\rtr}}$ is listed in Table~\ref{tablelambda_X^rtr}.
\begin{table}[H]
\begin{center}
\begin{tabular}{|c|c|c|c|c|c|c|c|c|}
\hline 
      & $0$ & $a$ & $b$ & $c$ & $d$ & $e$ & $f$ & $1$ \\ 
\hline 
$\lambda_{X^{\rtr}}$  & $0$ & $a$ & $a$ & $a$ & $d$ & $a$ & $a$ & $1$  \\ 
\hline
\end{tabular}
\caption{The interior operator $\lambda_{X^{\rtr}}$ of Example~\ref{t-norm_trellis_with_cycle}.}\label{tablelambda_X^rtr}
\end{center}
\end{table}
The t-norm $T^{[X^\rtr]}$ is given by Table~\ref{tableT^rtr}.
\begin{table}[H]
\begin{center}
\begin{tabular}{|c|c|c|c|c|c|c|}
\hline
$T^{[X^\rtr]}$ & $a$ & $b$ & $c$ & $d$ & $e$ & $f$ \\
\hline 
$a$ &  \cg $a$ & \cg $a$ & \cg $a$ & \cg $a$ & \cg $a$ & \cg $a$\\ 
\hline 
 $b$ & \cg $a$ & $a$ &  $a$ &  $a$ & \cg $a$ & $a$ \\ 
\hline 
$c$ & \cg $a$ &  $a$ & $a$ &  $a$ &  $a$ & $a$  \\ 
\hline 
$d$ &  \cg $a$ & $a$ &  $a$ & \cg $d$ & \cg $a$ & \cg $a$ \\ 
\hline 
$e$ &   \cg $a$ & \cg $a$ &  $a$ & \cg $a$ &  $a$ & $a$\\ \hline 
$f$ &   \cg $a$ &  $a$ &  $a$ & \cg $a$ & $a$ & $a$\\ 
\hline 
\end{tabular}
\caption{The t-norm $T^{[X^\rtr]}$ of Example~\ref{t-norm_trellis_with_cycle}.}\label{tableT^rtr}
\end{center}
\end{table}
\end{example} 
Combining Proposition~\ref{A_subtrellis_on_pseudo-chain} and Corollary~\ref{T_A_t-norm} leads to the following result.
\begin{proposition}\label{Special_case_pseud-chain}
Let $\mathbb{T}=(X,\unlhd,\meet,\join,0,1)$ be a bounded trellis and $A$ a subset of $X^\rtr$ containing $0$. If $\mathbb{T}$ is a pseudo-chain and $A$ is finite, then $T^{[A]}$ is a t-norm on~$\mathbb{T}$. Moreover, $T^{[A]}\unlhd T^{[X^\rtr]}$.
\end{proposition}

Combining Corollary~\ref{Eq_subsets_pseud-chain_and_modular},  Proposition~\ref{X^tr=X^ass_X^dis_sublattice} and Corollary~\ref{T_A_t-norm} leads to the following result.
\begin{corollary}\label{Special_case_modular}
Let $\mathbb{T}=(X,\unlhd,\meet,\join,0,1)$ be a bounded modular trellis and $\alpha \in\{\ass,\meetAss,\joinAss,\tr\}$. If $X^\alpha$ is finite, then $T^{[X^\alpha]}$ is a t-norm on~$\mathbb{T}$. Moreover, $T^{[X^\ass]}=T^{[X^\meetAss]}=T^{[X^\joinAss]}=T^{[X^\tr]}$.
\end{corollary}

In Corollary~\ref{Special_case_modular}, we did not consider the binary operation $T^{[X^\rtr]}$ 
since it is not necessarily a t-norm on~$\mathbb{T}$. This is due to the fact that $X^\rtr$ is not 
a sub-trellis of a bounded modular trellis $\mathbb{T}$ in general.
This is illustrated in the following example.

\begin{example}\label{counter}
Consider again the bounded modular trellis in Remark~\ref{Z_condition-not-satis}. One can verify that $X^\rtr=\{0,a,c,d,e,1\}$. It is not a $\wedge$-sub-trellis 
of~$\mathbb{T}$ since $c,d \in X^\rtr$ and $c \wedge d =b \notin X^\rtr$. Hence, it is not
a sub-trellis of~$\mathbb{T}$. Moreover, the binary operation $T^{[X^\rtr]}$ on~$\mathbb{T}$ given by Table~\ref{T^rtr_nt-norm} is not a t-norm since it is not increasing. Indeed, 
$c \unlhd e$, $d \unlhd e$, while $T^{[X^\rtr]}(c,d)=b \ntrianglelefteq e =T^{[X^\rtr]}(e,e)$.
\begin{table}[H] 
\begin{center}
\begin{tabular}{|c|c|c|c|c|c|}
\hline 
$T^{[X^\rtr]}$ & $a$& $b$ & $c$ & $d$ & $e$ \\ 
\hline 
$a$	   & \cg $a$& \cg $a$ & \cg $a$ & \cg $a$ & \cg $a$  \\  
\hline 
$b$	   & \cg $a$&  $a$ &  $a$ & $a$ & \cg $a$ \\ 
\hline 
$c$    & \cg $a$ &  $a$ &\cg  $c$ &\cg  $b$ & \cg $c$ \\ 
\hline 
$d$    & \cg $a$&  $a$ & \cg $b$ & \cg $d$ & \cg $d$  \\
\hline 
$e$    & \cg $a$& \cg $a$ & \cg $c$ & \cg $d$ & \cg $e$  \\
\hline 
\end{tabular}
\caption{The binary operation $T^{[X^\rtr]}$ of Example~\ref{counter}.}\label{T^rtr_nt-norm}
\end{center}
\end{table}
\end{example}

\section{Conclusion}\label{Conclusion}
In this paper, we have introduced the notion of a t-norm on bounded psosets and trellises and provided some basic examples. As the meet operation of a proper bounded trellis is not a t-norm, we have explored how we can construct a t-norm starting from a t-norm on an appropriate interior range.
We have provided illustrative examples by introducing interior operators whose range is an appropriate finite subset of the set of right-transitive elements of the given trellis. Similar results can be formulated for t-conorms on bounded psosets and trellises by considering the join operation, the set of left-transitive elements and closure operators. 
Since the set of right-transitive elements does not include non-trivial cycles, future work will explore how cycles can be involved. We anticipate that this study will open the door to various follow-up studies on t-norms and t-conorms on bounded psosets and trellises, and, more generally, on aggregation functions in the same setting.

\end{document}